\documentclass[1p]{elsarticle}
\usepackage{hyperref}
\usepackage{amssymb}
\usepackage{latexsym}
\usepackage{amscd}
\usepackage{amsthm}
\usepackage{amsfonts}
\usepackage{CJK}
\usepackage{CJKulem}
\usepackage{color}
\usepackage{colortbl}
\usepackage{fancyhdr}
\usepackage{indentfirst}
\usepackage{lastpage}
\usepackage{latexsym}
\usepackage{listings}
\usepackage{multirow}
\usepackage{mathrsfs}
\usepackage{placeins}
\usepackage{titlesec}
\usepackage{ulem}
\usepackage{graphicx} 
\usepackage{amsmath} 
\usepackage{amssymb} 
\usepackage{yfonts}
\usepackage{xypic}
\usepackage{bm}

\newdimen\AAdi%
\newbox\AAbo%
\def\AAk#1#2{\s_etbox\AAbo=\hbox{#2}\AAdi=\wd\AAbo\kern#1\AAdi{}}%
\def\AAr#1#2#3{\s_etbox\AAbo=\hbox{#2}\AAdi=\ht\AAbo\raise#1\AAdi\hbox{#3}}%
\font\tenmsb=msbm10 at 12pt \font\sevenmsb=msbm7 at 8pt
\font\fivemsb=msbm5 at 6pt
\newfam\msbfam
\textfont\msbfam=\tenmsb \scriptfont\msbfam=\sevenmsb
\scriptscriptfont\msbfam=\fivemsb
\def\Bbb#1{{\tenmsb\fam\msbfam#1}}
\newtheorem{thm}{Theorem}[section]
\newtheorem{lem}[thm]{Lemma}
\newtheorem{cor}[thm]{Corollary}

\newtheorem{pro}[thm]{Proposition}
\newtheorem{defi}[thm]{Definition}

\newcommand{\Section}[2]{\setcounter{equation}{0}
	\allowdisplaybreaks
	\section[#1]{#2}}

\def\f#1#2{\frac{#1}{#2}}

\def\mc#1{\mathcal{#1}}

\def\pd#1#2{\frac {\partial #1}{\partial #2}}

\def\td{\tilde}

\def\a{\alpha}
\def\be{\beta}

\def\p#1{\partial #1}

\def\de{\delta}

\def\ze{\zeta}
\def\ep{\varepsilon}

\def\G{\Gamma}
\def\g{\gamma}

\def\la{\lambda}
\def\La{\Lambda}
\def\om{\omega}
\def\Om{\Omega}
\def\th{\theta}

\def\si{\sigma}

\def\w{\wedge}

\def\R{\Bbb{R}}
\def\C{\Bbb{C}}

\def\lan{\langle}
\def\ran{\rangle}
\def\ra{\rightarrow}

\def\mb{\mathbf}

\def\Im{\text{Im }}

\begin{document}
	\title
	{On complete space-like stationary surfaces in $\R^{3,1}$ with graphical Gauss image}
	
	\author[f]{Li Ou}
	\ead{19110180007@fudan.edu.cn}
	\author[f]{Chuanmiao Cheng}
	\ead{17110180001@fudan.edu.cn}
	\author[f,m]{Ling Yang}
	\ead{yanglingfd@fudan.edu.cn}
	
	\address[f]{School of Mathematical Sciences, Fudan University, Shanghai, 200433, China}
	\address[m]{Shanghai Center for Mathematical Sciences, Shanghai, 200438, China}


	\begin{abstract}
       Concerning the value distribution problem for generalized Gauss maps, we not only
       generalize Fujimoto's theorem \cite{F} to complete space-like stationary surfaces in $\R^{3,1}$, but also
       estimate the upper bound of the number of exceptional values when the Gauss image lies in the graph of a rational
       function $f$ of degree $m$, which is determined by the number of solutions of $f(w)=\bar{w}$, showing
       a sharp contrast to Bernstein type results for minimal surfaces in $\R^4$. Moreover, we introduce the conception
       of conjugate similarity on $SL(2,\C)$ to classify all degenerate stationary surfaces (i.e. $m\leq 1$), and establish
       several structure theorems for complete stationary graphs in $\R^{3,1}$ from the viewpoint of the degeneracy of Gauss maps.

		\begin{keyword}
        Bernstein problem, space-like, stationary, Gauss map, exceptional value, degenerate, conjugate similar, graph.
		\MSC[2010] 53A10, 53C42, 53C45, 32H25, 30C15, 32A22, 51B20.\\
        This work was supported in part by NSFC (Grant No. 11622103).
		\end{keyword}
	\end{abstract}

	
	\maketitle
	
	\tableofcontents
	
	\renewcommand{\proofname}{\it Proof.}
	
	
	\Section{Introduction}{Introduction}

    The classical Berntein theorem \cite{Be} says {\it any entire minimal graph in $\R^3$ has to be an affine plane}. Afterwards,
    the Bernstein problem became one of the central topics in the theory of minimal surfaces, whose aim is to search geometric or
    analytic conditions forcing complete submanifolds of zero mean curvature in Euclidean spaces or other ambient spaces to be totally geodesic.

    For an oriented surface $M$ in $\R^3$, the canonical Gauss map $G:M\ra \mathbb{S}^2$ is conformal whenever $M$ is a minimal surface.
    Observing that $M$ is an entire graph if and only if the image of $M$ under $G$ lies in an open hemisphere, L. Nirenberg
    raised a well-known conjecture: {\it the Gauss image of any nonflat complete minimal surface in $\R^3$ has to be dense in $\Bbb{S}^2$}. Through the efforts
    of R. Osserman \cite{O,O1}, F. Xavier \cite{Xa} and H. Fujimoto \cite{F}, a stronger conclusion for the value distribution of Gauss maps was established: {\it any complete minimal surface whose Gauss map omits at least 5 points in $\Bbb{S}^2$ has to be an affine plane}. This result is best-possible, because of the known examples of complete minimal surfaces whose Gauss maps omit $k$ points with $0\leq k\leq 4$, constructed by
    K. Voss \cite{V}.

Following Gauss's idea, S. S. Chern \cite{Chern} and Estudillo-Romero \cite{E-R} introduced the conception of {\it generalized Gauss maps}
for oriented regular surfaces in $\R^n$ and oriented space-like surfaces in $n$-dimensional Minkowski space $\R^{n-1,1}$, respectively.
For $M\subset \R^n$ (or $\R^{n-1,1}$), the image of $p\in M$ under the Gauss map $G$ is the oriented tangent plane $T_p M\in \mb{G}_{n-2}^2$ (or $\mb{G}_{n-3,1}^{2}$) via parallel translation, where $\mb{G}_{n-2}^2$ (or $\mb{G}_{n-3,1}^{2}$) is the {\it Grassmannian manifold} (or {\it Lorentz Grassmanian manifold})
consisting of oriented planes in $\R^n$ (or $\R^{n-1,1}$) whose induced inner product is positive definite. Let
\begin{equation*}\aligned
Q_{n-2}&:=\{[(z_1,\cdots,z_n)]\in \Bbb{CP}^{n-1}:z_1^2+\cdots+z_n^2=0\},\\
Q_{n-3,1}&:=\{[(z_1,\cdots,z_n)]\in \Bbb{CP}^{n-1}:z_1^2+\cdots+z_{n-1}^2-z_n^2=0\},\\
Q_{n-3,1}^+&:=\{[(z_1,\cdots,z_n)]\in Q_{n-3,1}:|z_1|^2+\cdots+|z_{n-1}|^2-|z_n|^2>0\},
\endaligned
\end{equation*}
then
\begin{itemize}
\item $Q_{n-3,1}$ and $Q_{n-2}$ are both quadrics of $\Bbb{CP}^{n-1}$, and $\chi: [(z_1,\cdots,z_{n-1},z_n)]\in Q_{n-3,1}\mapsto [(z_1,\cdots,z_{n-1},iz_n)]\in Q_{n-2}$ gives a biholomorphic isomorphism between them;
\item $Q_{n-3,1}^+$ is an open subset of $Q_{n-3,1}$;
\item $\mb{G}_{n-2}^2$ ($\mb{G}_{n-3,1}^2$) can be identified with $Q_{n-2}$ ($Q_{n-3,1}^+$) via the canonical embedding into $\Bbb{CP}^{n-1}$.
\end{itemize}
Thereby, $\mb{G}_{n-3,1}^{2}$ can be seen as an open subset of $\mb{G}_{n-2}^2$.
The Gauss map $G$ is holomorphic provided that $M$ is a minimal (or {\it stationary}) surface in $\R^n$ (or $\R^{n-1,1}$), i.e. the mean curvature
field vanishes everywhere on $M$. Here $M\subset \R^{n-1,1}$ ($n\geq 4$) is called stationary since it is neither a local minimizer nor maximizer of the area
functional.

Especially for a minimal surface $M$ in $\R^4$, since $Q_2$ is conformally equivalent to $\Bbb{S}^2\times \Bbb{S}^2$
(see e.g. \cite{Chern,H-O}), the Gauss map $G$ can be written as $(\psi_1,\psi_2)$, where both components are meromorphic functions
on $M$. From this viewpoint, H. Fujimoto \cite{F} obtained the following Bernstein type results:
\begin{thm}(\cite{F})\label{be}
Let $M$ be a nonflat complete minimal surface in $\R^4$, $(\psi_1,\psi_2)$ be the Gauss map of $M$, and $q_i$ be the number
of exceptional values of $\psi_i$ ($i=1$ or $2$).
\begin{itemize}
\item If neither $\psi_1$ nor $\psi_2$ are constant, then $\min\{q_1,q_2\}\leq 3$ or $q_1=q_2=4$.
\item If either $\psi_1$ or $\psi_2$ is constant, then $q_i\leq 3$ for the nonconstant $\psi_i$.
\end{itemize}
\end{thm}

Similarly, $Q_{1,1}^+\subset Q_{1,1}\cong \Bbb{S}^2\times \Bbb{S}^2$ enables us to express the Gauss map $G$ of a space-like stationary surface $M\subset \R^{3,1}$ as $(\psi_1,\psi_2)$, but here the two components should satisfy an additional constraint $\psi_2\neq \bar{\psi}_1$. As pointed out by Ma-Wang-Wang \cite{M-W-W}, $\psi_1(p)$ and $\psi_2(p)$ correspond to the 2 null vectors $\mb{y},\mb{y}^*$ in the normal plane $N_p M$, and $\psi_2\neq \bar{\psi}_1$ is equivalent to
the linear independence of $\mb{y}$ and $\mb{y}^*$. In the same paper, the authors asked whether we can estimate the number of exception values
of $\psi_1$ and $\psi_2$ for complete nonflat stationary surfaces in $\R^{3,1}$, as H. Fujimoto \cite{F} has done.

In this paper, we utilize the fundamental tool of Weierstrass type representations to construct a one-to-one correspondence
$M\leftrightarrow M^*$ between all simply-connected {\it generalized stationary surfaces} in $\R^{3,1}$ and all simply-connected {\it generalized minimal surfaces} in $\R^4$, where both $M$ and $M^*$ may have singularities, such that $\chi:Q_{1,1}\ra Q_{2}$ gives a one-to-one correspondence between the Gauss images of $M$ and $M^*$. Moreover, $M^*$ becomes a complete minimal surface whenever the metric of $M$ is positive-definite everywhere and complete.
By this argument, we generalize Fujimoto's theorem \cite{F} and answer the above problem raised by Ma-Wang-Wang \cite{M-W-W}:
\begin{thm}\label{be0}
Let $M$ be a nonflat complete space-like stationary surface in $\R^{3,1}$, $(\psi_1,\psi_2)$ be the Gauss map of $M$, and $q_i$ be the number
of exceptional values of $\psi_i$ ($i=1$ or $2$). If neither $\psi_1$ nor $\psi_2$ are constant, then $\min\{q_1,q_2\}\leq 3$ or $q_1=q_2=4$.
\end{thm}

It worths to note that, any minimal surface $M\subset \R^3$ is not only a minimal surface in $\R^4$, but also a stationary surface
in $\R^{3,1}$,
and in both situations,  $q_1,q_2$ are both equal to
the number of values in $\Bbb{S}^2$ that the classical Gauss map does not take. Therefore, the known examples of minimal surfaces in $\R^3$
prevent us to improve Theorem \ref{be}-\ref{be0} without additional assumptions. And the conclusions of Theorem \ref{be} and Theorem \ref{be0}
seem to be 'same' in the general setting.
	
However, under a premise that the Gauss image of $M$ totally lies in a prescribed non-extendable complex curve $\G\subset Q_{1,1}$,
we may establish Bernstein type theorems whose conditions are strictly weaker than the corresponding theorems for complete minimal surfaces
in $\R^4$ whose Gauss images lie in $\chi(\G)$. Firstly, we choose hyperplanes of $Q_{1,1}$, whose corresponding stationary surfaces are said to be {\it degenerate}, to be the candidates of $\G$, due to several reasons as follows:
\begin{itemize}
\item Minimal surfaces in $\R^3$ and maximal surfaces in $\R^{2,1}$ can both be seen as degenerate stationary surfaces in $\R^{3,1}$; however,
there exists no nonflat complete maximal surface in $\R^{2,1}$, in contrast to an abundance of examples of complete minimal surfaces in $\R^3$.
(In fact, as shown by E. Calabi \cite{Ca} and Cheng-Yau \cite{C-Y}, any complete maximal space-like hypersurface in $\R^{n,1}$ has to be affine linear.)
\item Hoffman-Osserman \cite{H-O} pointed out, a degenerate minimal surface in $\R^4$ is either a complex curve with respect to
an orthogonal complex structure on $\R^4$, or lies in a $1$-parameter family of degenerate ones determined by a minimal surface in $\R^3$.
On the contrary, there exist at least 3 different types of degenerate stationary surfaces in $\R^{3,1}$, having quite distinct geometric properties; this is the main topic of \cite{A-V1} by Asperti-Vilhena.
\item As shown by R. Osserman \cite{O2} and Ma-Wang-Yang \cite{M-W-Y}, any entire minimal (stationary) graph in $\R^4$ ($\R^{3,1}$) over $\R^2$
has to be degenerate. Thereby, it is an interesting problem to further investigate the relationship between the graphical conditions and
the degeneracy of Gauss maps.
\end{itemize}

Our investigation on degenerate stationary surfaces begins with the classification of hyperplanes in $Q_{1,1}$. Let
$H_A$ and $H_B$ be 2 hyperplanes with respect to 2 elements $[A],[B]$ in the complex projective space, respectively (see (\ref{hp1})),
then $H_A$ and $H_B$ are equivalent under the action of the Lorentz transformation group $SO^+(3,1)$ if and only if
$[A],[B]$ lie in the same orbit of $SO^+(3,1)$. Thus, the classification problem of hyperplanes reduces to
determining the orbit space of $SO^+(3,1)$-action. This is a traditional idea originated with Hoffman-Osserman \cite{H-O} to
dealing with the classification of hyperplanes in $Q_2$. In this paper, we adopt an alternative approach as follows.
Let $M\subset \R^{3,1}$ be a degenerate stationary surface whose Gauss image lies in $H_A$,
then
$$[A]\in Q_{1,1}\Leftrightarrow\ \text{Either } \psi_1\text{ or }\psi_2\text{ is constant}\Leftrightarrow\ K\equiv 0$$
$$\Leftrightarrow\ M\text{ is 2-degenerate}\ \Leftrightarrow\ M\text{ lies in }\R^{2,0}.$$
(Here $K$ is the Gauss curvature of $M$, $M$ is {\it $2$-degenerate} means its Gauss image lies in the intersection of 2 distinct hyperplanes,
and $\R^{2,0}$ is the direct sum of a Euclidean plane and a null line, see Theorem \ref{t_r2} for details.)
Otherwise, $H$ is an entire graph of a M\"{o}bius transformation over $\Bbb{S}^2$ (see Proposition \ref{hp2}), i.e.
there exists $S=\begin{pmatrix} a & b\\c & d\end{pmatrix}\in SL(2,\C)$, such that
$\psi_2=\mc{M}_S(\psi_1):=\f{a\psi_1+b}{c\psi_1+d}$. The constraint
$\psi_2\neq \bar{\psi}_1$ implies $\psi_1$ cannot take values in
$$E_S=\{w:\mc{M}_S(w)=\bar{w}\},$$
which is completely determined by the {\it conjugate eigenvectors} of $S$ (see Proposition \ref{conj2}).
Moreover, the graph of $\mc{M}_{S_1}$ are equivalent to the graph of $\mc{M}_{S_2}$ under $SO^+(3,1)$-action if and only if
$S_1$ and $S_2$ are {\it conjugate similar} to each other (see Proposition \ref{conj3}).
 By carefully examining the canonical forms
of the conjuate similarity in $SL(2,\C)$, we classify all such graphs in $Q_{1,1}$ in Theorem \ref{class3}, and hence an arbitrary nonflat degenerate stationary surface $M\subset \R^{3,1}$
must belong to one and only one of the following categories:
\begin{itemize}
\item $M$ is a minimal surface in $\R^3\Leftrightarrow [A]$ is {\it totally real} and $E_S=\emptyset$;
\item $M$ is a maximal surface in $\R^{2,1}\Leftrightarrow$ $|E_S|=\infty$;
\item $M$ is a {\it degenerate stationary surface of hyperbolic type} $\Leftrightarrow$ $P_A$ is space-like $\Leftrightarrow$ $|E_S|=2$;
\item  $M$ is a {\it degenerate stationary surface of elliptic type} $\Leftrightarrow$ $P_A$ is time-like $\Leftrightarrow$ $[A]$ is non-totally real and $|E_S|=\emptyset$;
\item  $M$ is a {\it degenerate stationary surface of parabolic type} $\Leftrightarrow$ $P_A$ is light-like $\Leftrightarrow$ $|E_S|=1$.
\end{itemize}
(Here $[A]$ is totally real means the real part and the imaginary part of $A$ is linear dependent, otherwise $P_A$ is defined as the plane spanned by  the 2 vectors.)

Theorem \ref{class3} enables us to get Weierstrass type representations for all types of degenerate stationary surfaces in $\R^{3,1}$.
Afterwards, by the arguments of R. Osserman \cite{O-Survey}, Al\'ias-Palmer \cite{A-P} and H. Fujimoto \cite{F}, we arrive at
the following results on the number of exception values of the Gauss maps:
\begin{thm}\label{be4}
Let $M$ be a complete degenerate space-like stationary surface in $\R^{3,1}$, which is not an affine plane, $(\psi_1,\psi_2)$
be the Gauss map of $M$. Assume $\psi_i$ omits $q_i$ points ($i=1$ or $2$), then
\begin{itemize}
\item If $M$ is 2-degenerate, then $q_i\in \{1,2\}$ for the nonconstant $\psi_i$. 
\item If $M$ is a nonflat degenerate stationary surface of hyperbolic type, then $q_1=q_2=2$. 
\item If $M$ is a minimal surface in $\R^3$ or a degenerate stationary surface of elliptic type, then
$q_1=q_2\in \{0,1,2,3,4\}$. 
\item If $M$ is a degenerate stationary surface of parabolic type, then $q_1=q_2\in \{1,2,3\}$. 
\end{itemize}
\end{thm}
It worths to note that, except for the case $E_S=\emptyset$, the conditions of the above Bernstein type results are
strictly weaker than corresponding theorems for degenerate minimal surfaces in $\R^4$. (In fact, let $M\subset \R^{3,1}$ be a nonflat degenerate
stationary surface of whether hyperbolic or parabolic type, $M^*\subset \R^4$ is a degenerate minimal surface that can be deformed into a minimal surface in $\R^3$; we can prove $q_1=q_2\in \{0,1,2,3,4\}$ for such surfaces provided the nonflatness and completeness, as in the proof of Theorem \ref{be2}.)

Meanwhile, we study complete stationary graphs in $\R^{3,1}$. By showing the existence of global isothermal parameters as in \cite{O2},
we prove a universal conclusion for whatever kind of entire graphs: {\it any entire space-like stationary graph over a space-like (or time-like, light-like) plane should be
a complete degenerate stationary surface of hyperbolic (or elliptic, parabolic) type} (See Theorem \ref{grph}, Theorem \ref{graph1} and Theorem \ref{graph4}). On the other hand, the 3 types of graphs have different features
as follows:
\begin{itemize}
\item Each complete degenerate surface in $\R^{3,1}$ of hyperbolic type (including the 2-degenerate case) has to be an entire graph over a space-like plane (see Theorem \ref{t_r2}
and Theorem \ref{grph}). However, this conclusion does not hold for elliptic and parabolic cases.
\item There exists a constellation of complete stationary graphs over domains of a time-like plane with finite total Gauss curvature, all of which are degenerate stationary surfaces of elliptic type induced by rational functions satisfying the following condition: both $h$ and $h^{-1}$ are derivatives of rational functions (see Theorem \ref{graph-time}).
    On the contrary, there exist no nonflat complete degenerate stationary surface of whether hyperbolic or parabolic type with finite total Gauss curvature (see Theorem \ref{grph} and Remarks of Theorem \ref{graph3}).
\item In contrast to the hyperbolic and elliptic cases, there exist nonflat algebraic entire graphs over a light-like plane. Moreover, these surfaces
can be characterized as complete algebraic degenerated stationary surfaces of parabolic type (see Theorem \ref{grph}, Theorem \ref{graph1} and Theorem \ref{graph3}).
\end{itemize}

Next, we generalize Theorem \ref{be4} to complete stationary surfaces in $\R^{3,1}$ with {\it rational graphical Gauss image},
showing what it reveals are not isolated phenomena:
\begin{thm}
Let $M$ be a nonflat complete space-like stationary surface in $\R^{3,1}$, whose Gauss map $(\psi_1,\psi_2)$ satisfies
$\psi_2=f(\psi_1)$ with $f$ a rational function of degree $m$, then the number $q_1$ of the exceptional values of $\psi_1$
should satisfy $|E_f|\leq q_1\leq m-|E_f|+3$, where $E_f$ consists of all solutions of $f(w)=\bar{w}$.
\end{thm}
For $m\geq 2$, the proof of the above conclusion relies on the discreteness and nonempty of $E_f$. Moreover, in virtue of
the theory of differential topology, we show $E_f$ has at least $m-1$ points and study the local behavior of $f$ near
each point of $E_f$ (see Proposition \ref{ef} and its proof), and it follows that:
\begin{thm}
Let $M$ be a complete space-like stationary surface in $\R^{3,1}$ whose Gauss image lies in the graph of a rational function
$f$ of degree $m\geq 2$, then each one of the following additional conditions forces $M$ to be an affine space-like plane:
$$(a)\ m\geq 6;\quad (b)\ m\leq 5\text{ and }|E_f|> \f{m+3}{2};\quad (c)\ \int_M |K| dA_M<+\infty. $$
\end{thm}

The above interesting results inspire us to study complete space-like stationary surfaces in $\R^{3,1}$ whose Gauss image
lies in a prescribed algebraic curve $\G\subset Q_{1,1}$ in the future. We conjecture such $M$ has to affine linear whenever $\G$ is sufficiently
complicated.

	\bigskip\bigskip

	\Section{The geometry of the Lorentz Grassmannian manifold and its hyperplanes}{The geometry of the Lorentz Grassmannian manifold and its hyperplanes}\label{L_G}
	
\subsection{The conformal and metric structures of $\mb{G}_{1,1}^2$}
	Let $\R^{3,1}$ be the $4$-dimensional Minkowski space. The Minkowski inner product for $\mb{u}=(u_1,u_2,u_3,u_4)$
	and $\mb{v}=(v_1,v_2,v_3,v_4)$ is given by
	\begin{equation}
	\lan \mb{u},\mb{v}\ran:=u_1v_1+u_2v_2+u_3v_3-u_4v_4.
	\end{equation}
	Denote by $\mb{G}_{1,1}^2$ the Lorentz Grassmannian manifold consisting of all oriented space-like 2-plane in $\R^{3,1}$.
	Given $\Pi\in \mb{G}_{1,1}^2$, $\mb{u}\w \mb{v}$ is the {\it Pl\"{u}cker coordinate} of $\Pi$ with $\{\mb{u},\mb{v}\}$ an oriented orthonormal
	basis of $\Pi$. The canonical pseudo-Riemannian metric $g$ on $\mb{G}_{1,1}^2$ is induced by the Pl\"{u}cker embedding
	from $\mb{G}_{1,1}^2$ into $\La_{3,1}^2$, which is a linear space of all $2$-vectors in $\R^{3,1}$ and the inner product
	\begin{equation}
	\lan \mb{u}_1\w \mb{v}_1,\mb{u}_2\w \mb{v}_2\ran:=\begin{vmatrix}
	\lan \mb{u}_1,\mb{v}_1\ran & \lan \mb{u}_1,\mb{v}_2\ran\\
	\lan \mb{u}_2,\mb{v}_1\ran & \lan \mb{u}_2,\mb{v}_2\ran
	\end{vmatrix}
	\end{equation}
	is induced by the Minkowski inner product on $\R^{3,1}$.
	
	On the other hand, for $\Pi=\text{span}\{\mb{u},\mb{v}\}$, the complex vector $\mb{z}=\mb{u}-i\mb{v}$ assigns a point in
	$\C^{3,1}$, which is the complexification of $\R^{3.1}$ equipped with
	a complex bilinear form
	\begin{equation}
	\lan \mb{u}_1-i\mb{v}_1,\mb{u}_2-i\mb{v}_2\ran:=\lan \mb{u}_1,\mb{u}_2\ran-\lan \mb{v}_1,\mb{v}_2\ran-i\lan \mb{u}_1,\mb{v}_2\ran-i\lan \mb{v}_1,\mb{u}_2\ran.
	\end{equation}
	Namely, for $\mb{z}=(z_1,z_2,z_3,z_4), \mb{w}=(w_1,w_2,w_3,w_4)\in \C^{3,1}$,
	\begin{equation}
	\lan \mb{z},\mb{w}\ran=z_1w_1+z_2w_2+z_3w_3-z_4w_4.
	\end{equation}
	And
	\begin{equation}
	(\mb{z},\mb{w}):=\lan \mb{z},\bar{\mb{w}}\ran=z_1\bar{w}_1+z_2\bar{w}_2+z_3\bar{w}_3-z_4\bar{w}_4
	\end{equation}
	defines a pseudo-Hermite inner product. Noting that another oriented orthonormal basis $\{\cos\th\ \mb{u}+\sin \th\ \mb{v},-\sin\th\ \mb{u}+\cos\th\ \mb{v}\}$
	of $\Pi$ corresponds to $e^{i\th}\mb{z}$,
	\begin{equation}
	i: \Pi=\text{span}\{\mb{u},\mb{v}\}\in \mb{G}_{1,1}^2\mapsto [\mb{z}]=[\mb{u}-i\mb{v}]\in \Bbb{CP}^{2,1}
	\end{equation}
	is well-defined and injective, where $\Bbb{CP}^{2,1}$ is the 3-dimensional complex projective space equipped with the Fubini-Study type metric
	\begin{equation}
	h_{FS}:=\f{2\lan \mb{z}\w d\mb{z},\bar{\mb{z}}\w d\bar{\mb{z}}\ran}{\lan \mb{z},\bar{\mb{z}}\ran^2}.
	\end{equation}
	Since $\lan \mb{u},\mb{u}\ran=\lan \mb{v},\mb{v}\ran=1$, $\lan \mb{u},\mb{v}\ran=0$, we have
	\begin{equation}\label{con1}
	\lan \mb{z},\mb{z}\ran=\lan \mb{u},\mb{u}\ran-\lan \mb{v},\mb{v}\ran-2i\lan \mb{u},\mb{v}\ran=0
	\end{equation}
	and
	\begin{equation}\label{con2}
	(\mb{z},\mb{z})=\lan \mb{z},\bar{\mb{z}}\ran=\lan \mb{u},\mb{u}\ran+\lan \mb{v},\mb{v}\ran=2>0.
	\end{equation}
	Hence $i$ gives a one-to-one correspondence between $\mb{G}_{1,1}^2$ and $Q_{1,1}^+\subset Q_{1,1}$, where
	\begin{eqnarray}
	&Q_{1,1}:=\{[\mb{z}]\in \Bbb{CP}^{2,1}:\lan \mb{z},\mb{z}\ran=0\},\\
	&Q_{1,1}^+:=\{[\mb{z}]\in Q_{1,1}:  (\mb{z},\mb{z})>0\}.
	\end{eqnarray}
	
Differentiating both sides of (\ref{con1}) and (\ref{con2}) yields
	$\lan \mb{z},d\mb{z}\ran=\lan \bar{\mb{z}},d\bar{\mb{z}}\ran=0$ and $\lan \mb{z},d\bar{\mb{z}}\ran+\lan \bar{\mb{z}},d\mb{z}\ran=0$, which enable us to obtain
	\begin{equation}\label{g1}
	\aligned
	g&=\lan d(\mb{u}\w \mb{v}),d(\mb{u}\w \mb{v})\ran=\left\lan -\f{i}{2}d(\mb{z}\w \bar{\mb{z}}),-\f{i}{2}d(\mb{z}\w \bar{\mb{z}})\right\ran\\
	&=\f{1}{4}(\lan \mb{z},d\bar{\mb{z}}\ran^2+\lan \bar{\mb{z}},d\mb{z}\ran^2+2\lan \mb{z},\bar{\mb{z}}\ran\lan d\mb{z},d\bar{\mb{z}}\ran)\\
	&=\f{1}{2}(\lan \mb{z},\bar{\mb{z}}\ran\lan d\mb{z},d\bar{\mb{z}}\ran-\lan \mb{z},d\bar{\mb{z}}\ran\lan \bar{\mb{z}},d\mb{z}\ran)\\
	&=\f{1}{2}\lan \mb{z}\w d\mb{z},\bar{\mb{z}}\w d\bar{\mb{z}}\ran=\f{2\lan \mb{z}\w d\mb{z},\bar{\mb{z}}\w d\bar{\mb{z}}\ran}{\lan \mb{z},\bar{\mb{z}}\ran^2}.
	\endaligned
	\end{equation}
	This means $i:(\mb{G}_{1,1}^2,g)\ra (\Bbb{CP}^{2,1},h_{FS})$ is an isometric embedding, i.e. $\mb{G}_{1,1}^2$ and
$Q_{1,1}^+$ are isometrically equivalent.

	Following the argument of \cite{H-O}, we study the conformal structure of $Q_{1,1}$ and $Q_{1,1}^+$.
	\begin{pro}\label{com_str}
		For $Q_{1,1}$ and $Q_{1,1}^+$, we have:
		\begin{enumerate}
			\item[(1)] $Q_{1,1}$ is biholomorphic to $\C^*\times \C^*$, with $\C^*:=\C\cup \{\infty\}$
			the extended complex plane.
			\item[(2)] $Q_{1,1}^+$ is biholomorphic to $\{(w_1,w_2)\in \C^*\times \C^*:w_2\neq \bar{w}_1\}$, where the complex conjugate
of $\infty$ is defined to be $\infty$.
		\end{enumerate}
	\end{pro}
	
	\begin{proof}
		For $[\mb{z}]=[(z_1,z_2,z_3,z_4)]\in Q_{1,1}$, $z_1^2+z_2^2+z_3^2-z_4^2=0$ implies
		\begin{equation}\label{Q31}
		(z_1+iz_2)(z_1-iz_2)=(z_4-z_3)(z_3+z_4).
		\end{equation}
		Let
		\begin{equation}
		\Om:=\{[\mb{z}]\in Q_{1,1}:z_3+z_4\neq 0\},
		\end{equation}
		then for each $[\mb{z}]\in \Om$, (\ref{Q31}) implies
		\begin{equation}\label{wl20}
		\f{z_1+iz_2}{z_3+z_4}=\f{z_4-z_3}{z_1-iz_2},\qquad \f{z_1-iz_2}{z_3+z_4}=\f{z_4-z_3}{z_1+iz_2}.
		\end{equation}
		Now we define $\Psi: [\mb{z}]\in \Om\mapsto (w_1,w_2)\in \C\times \C$
		\begin{equation}\label{w12}
		w_1=\f{z_1+iz_2}{z_3+z_4},\quad w_2=\f{z_1-iz_2}{z_3+z_4}.
		\end{equation}
		Then a straightforward calculation based on (\ref{wl20}) and (\ref{w12}) shows
		\begin{equation}\label{inv_F}
		[\mb{z}]=[w_1+w_2,-i(w_1-w_2),1-w_1w_2,1+w_1w_2].
		\end{equation}
		Therefore, $\Psi$ is a biholomorphism between $\Om$ and $\C\times \C$, and the inverse of $\Psi$ is given by (\ref{inv_F}).
		
		It is easily-seen that $\Psi$ extends continuously to a bijection between $Q_{1,1}$ and $\C^*\times \C^*$, such that
		\begin{itemize}
			\item $w_1=\f{z_4-z_3}{z_1-iz_2},w_2=\infty$ whenever $z_3+z_4=z_1+iz_2=0,z_1-iz_2\neq 0$;
			\item $w_1=\infty,w_2=\f{z_4-z_3}{z_1+iz_2}$ whenever $z_3+z_4=z_1-iz_2=0,z_1+iz_2\neq 0$;
			\item $w_1=\infty,w_2=\infty$ whenever $z_3+z_4=z_1+iz_2=z_1-iz_2=0.$
		\end{itemize}
		Equivalently,
		\begin{itemize}
			\item $\mb{z}=[(1,i,-w_1,w_1)]$ whenever $w_1\in \C,w_2=\infty$;
			\item $\mb{z}=[(1,-i,-w_2,w_2)]$ whenever $w_1=\infty, w_2\in \C$;
			\item $\mb{z}=[(0,0,-1,1)]$ whenever $w_1=\infty, w_2=\infty$.
		\end{itemize}
		Around $\infty$, by using variables $\ze_1:=\f{1}{w_1}$, $\ze_2:=\f{1}{w_2}$, we conclude that $\Psi$ is a biholomorphism between
		$Q_{1,1}$ and $\C^*\times \C^*$.
		
		By (\ref{inv_F}),
		\begin{equation}
		(\mb{z},\mb{z})=|w_1+w_2|^2+|-i(w_1-w_2)|^2+|1-w_1w_2|^2-|1+w_1w_2|^2=2|w_2-\bar{w}_1|^2.
		\end{equation}
		Hence $[\mb{z}]\in Q_{1,1}^+$ if and only if $w_2\neq \bar{w}_1$. This complete the proof of (2).
		
	\end{proof}
	
    Let $\mc{P}:\Bbb{S}^2\ra \C^*$ be the stereographic projection from the point $(0,0,-1)$, then
		$$\mc{P}^{-1}(w)=\left(\f{2\text{Re }w}{1+|w|^2},\f{2\text{Im }w}{1+|w|^2},\f{1-|w|^2}{1+|w|^2}\right)\qquad \forall w\in \C.$$
		Denote by $\mc{L}:=\{\mb{u}\in \R^{3,1}:\lan \mb{u},\mb{u}\ran=0\}$ the light cone of $\R^{3,1}$, then
		$\mb{x}\mapsto (\mb{x},1)$ gives a one-to-one correspondence between $\Bbb{S}^2$ and $\mc{L}\cap \{x_4=1\}$. Given $\Pi=\text{span}\{\mb{u},\mb{v}\}\in \mb{G}_{1,1}^2$, we can check that
		\begin{equation}\label{geo_m}
		\mb{y}:=(\mc{P}^{-1}(w_1),1),\quad \mb{y}^*:=(\mc{P}^{-1}(\bar{w}_2),1)
		\end{equation}
		is the unique pair in $\mc{L}\cap \{x_4=1\}$ satisfying $\mb{y},\mb{y}^*\perp \Pi$ and $\det(\mb{u},\mb{v},\mb{y},\mb{y}^*)>0$, and $w_2\neq \bar{w}_1$
		is equivalent to saying that $\mb{y}\neq \mb{y}^*$. This is the geometric meaning of $w_1$ and $w_2$, as shown in \cite{Wang} and  \cite{M-W-W}.

	Now we express the metric of $Q_{1,1}^+$ in terms of $w_1$ and $w_2$. Differentiating both sides of (\ref{inv_F}), we get
	\begin{equation}\label{dz}
	d\mb{z}=(dw_1+dw_2,-i(dw_1-dw_2),-(w_2dw_1+w_1dw_2),w_2dw_1+w_1dw_2).
	\end{equation}
	Substituting (\ref{inv_F}) and (\ref{dz}) into (\ref{g1}), we arrive at the following formula
	\begin{equation}\label{g2}
	g=\text{Re}\left[\f{4d\bar{w}_1dw_2}{(\bar{w}_1-w_2)^2}\right].
	\end{equation}

	
	\subsection{The graphical structure of hyperplanes in $Q_{1,1}$}
	For an arbitrary $[A]\in \Bbb{CP}^{2,1}$,
	\begin{equation}\label{hp1}
   \aligned
	H_A:=\{[\mb{z}]\in Q_{1,1}:\lan A,\mb{z}\ran=0\}\\
    H_A^+:=\{[\mb{z}]\in Q_{1,1}^+:\lan A,\mb{z}\ran=0\}
    \endaligned
	\end{equation}
are hyperplanes
	in $Q_{1,1}$ and $Q_{1,1}^+$ with respect to $[A]$, respectively.
	
	\begin{pro}\label{hp2}
		For $S\subset Q_{1,1}$, we have:
		\begin{enumerate}
			\item [(1)] $S=H_A$ with $[A]\notin Q_{1,1}$ if and only if $\Psi(S)$ is the entire graph of
			a M\"{o}bius transformation $\mc{M}:\C^*\ra \C^*$.
			\item [(2)] $S=H_A$ with $[A]\in Q_{1,1}$ if and only if $\Psi(S)=\{(w_1,w_2)\in \C^*\times \C^*:w_1=c_1\text{ or }w_2=c_2\}$
with $c_1,c_2\in \C^*$.
			Especially $[A]\in Q_{1,1}^+$ if and only if $c_2\neq \bar{c}_1$.
		\end{enumerate}
	\end{pro}
	
	\begin{proof}
		If $[A]:=[(a_1,a_2,a_3,a_4)]\notin Q_{1,1}$, then
		\begin{equation}\begin{vmatrix}
		a_1-ia_2 & a_3-a_4\\
		a_4+a_3 & -(a_1+ia_2)
		\end{vmatrix}=-(a_1^2+a_2^2+a_3^2-a_4^2)\neq 0.
\end{equation}
		In conjunction with (\ref{inv_F}) and (\ref{hp1}), $[\mb{z}]=[(z_1,z_2,z_3,z_4)]\in H_A$ if and only if
		\begin{equation}
		w_2=\f{(a_1-ia_2)w_1+a_3-a_4}{(a_4+a_3)w_1-(a_1+ia_2)},
		\end{equation}
		i.e. $\Psi(H_A)$ is the graph of a M\"{o}bius transformation. Conversely, if $\Psi(S)$ is the graph of
		\begin{equation}
\mc{M}:w_1\in \C^*\mapsto w_2=\f{aw_1+b}{cw_1+d}\in \C^*
\end{equation}
		with $ad-bc\neq 0$, then we can show $S=H_A$ with
		\begin{equation}
A=(a-d,i(a+d),b+c,c-b),
\end{equation}
		and $[A]\notin Q_{1,1}$.
		
On the other hand, again using (\ref{inv_F}) and (\ref{hp1}), a straightforward calculation shows:
		 whenever $S=H_A$ with $[A]\in Q_{1,1}$,
		\begin{equation}\label{psiH}
		\Psi(S)=\{(w_1,w_2)\in \C^*\times \C^*:w_1=c_1\text{ or }w_2=c_2\}
		\end{equation}
		with $(c_1,c_2)=\Psi([A])$. Here $c_2\neq \bar{c}_1$
		if and only if $[A]\in Q_{1,1}^+$. Conversely, for each $S\subset Q_{1,1}$ satisfying (\ref{psiH}), it is easy to check
		that $S=H_A$ with $[A]=\Psi^{-1}(c_1,c_2)$.
		
	\end{proof}
	
This proposition along with an easily-seen fact that $\mc{M}_1(w)=\mc{M}_2(w)$ has 1 or 2 solutions in $\C^*$
for any 2 distinct M\"{o}bius transformations
enable us to derive the following conclusion:

	\begin{cor}\label{hp3}
		Let $E\subset Q_{1,1}$ be the intersection of 2 distinct hyperplanes, then one and only one of the following 3 cases occurs:
		\begin{enumerate}
			\item [(1)] $1\leq |E|\leq 2$;
			\item [(2)] $\Psi(E)=\{(c_1,w_2):w_2\in \C^*\}$ with $c_1\in \C^*$;
			\item [(3)] $\Psi(E)=\{(w_1,c_2):w_1\in \C^*\}$ with $c_2\in \C^*$.
		\end{enumerate}
	\end{cor}
	
		
		
		
		

	
	\subsection{The $SO^+(3,1)$-action on $\Bbb{CP}^{2,1}$}
	
	To examine more closely the properties of hyperplanes in $Q_{1,1}$, we shall try to search
	the representative hyperplanes under the actions of the Lorentz transformation groups, as in \cite{H-O,A-V1}.
	
		Let $O(3,1)$ be the Lie group of linear transformations of $\R^{3,1}$ which keeps the Minkowski inner product invariant,
		$SO(3.1):=\{\si\in O(3,1):\det\si=1\}$, and
		$SO^+(3,1):=\{\si\in SO(3,1):(\si(\ep_4))_4>0\}$, with $\ep_4:=(0,0,0,1)$.
	Due to the definition of $H_A$ and $O(3,1)$, we have
\begin{equation}
\si(H_A)=H_{\si(A)},\quad \si(H_A^+)=H_{\si(A)}^+
\end{equation}
for each $\si\in O(3,1)$, and it follows that:
	
	\begin{lem}\label{sim}
		Any 2 hyperplanes $H_A,H_B\subset Q_{1,1}$ are equivalent under the action of $O(3,1)$ (or $SO(3,1)$, $SO^+(3,1)$), if and only if
$[A]$ and $[B]$ lie in the same orbit of $O(3,1)$ (or $SO(3,1)$, $SO^+(3,1)$) acting on $\Bbb{CP}^{2,1}$.
And the same conclusion still
		holds for 2 hyperplanes in $Q_{1,1}^+$.
	\end{lem}
	
	$[A]\in \Bbb{CP}^{2,1}$ is called {\it totally real}, whenever $[A]=[\mb{u}_0]$ for a real vector $\mb{u}_0$; and the corresponding
 $H_A$ ($H_A^+$) is called a {\it totally real hyperplane} in $Q_{1,1}$ ($Q_{1,1}^+$). Obviously, the set of all totally real (or non-totally real) elements (or hyperplanes) in $\Bbb{CP}^{2,1}$ (or $Q_{1,1}$) is invariant under
 the action of $O(3,1)$.
Via examining the type of $\mb{u}_0$ (for the totally real case) or the plane spanned by $\text{Re }A$ and $\text{Im }A$ (for the
non-totally real case),
we can give the classification for
	all elements (or hyperplanes) in $\Bbb{CP}^{2,1}$ (or $Q_{1,1}$) as follows:
	
	\begin{thm}\label{class1}
For any $[A]\in \Bbb{CP}^{2,1}$, the orbits of $[A]$ under the action of $O(3,1)$, $SO(3,1)$ and
$SO^+(3,1)$ are just the same one, and we denote $[A]\sim [B]$, $H_A\sim H_B$ for each $[B]$ lying in this orbit.
Moreover, each $[A]=[\text{Re }A+i\text{Im }A]:=[X+iY]$, as well as the corresponding hyperplane $H_A\subset Q_{1,1}$, can be classified to one and only one of the following categories:
		\begin{itemize}
			\item  {\bf Totally real type:} both $X$ and $Y$ are multiples of $\mb{u}_0\in \R^{3,1}$.
            \begin{enumerate}
            \item [(a)] $[A]\sim [\mb{u}_S]$ with $\mb{u}_S:=\ep_1=(1,0,0,0)$ if and only if $\mb{u}_0$ is space-like, and then
$H_A\sim \{[\mb{z}]\in Q_{1,1}:z_1=0\}$.
		    \item [(b)] $[A]\sim [\mb{u}_T]$ with $\mb{u}_T:=\ep_4=(0,0,0,1)$ if and only if $\mb{u}_0$ is time-like, and then
$H_A\sim \{[\mb{z}]\in Q_{1,1}:z_4=0\}$.
			\item [(c)] $[A]\sim [\mb{u}_N]$ with $\mb{u}_N:=\ep_3+\ep_4=(0,0,1,1)$ if and only if $\mb{u}_0$ is null, and then
$H_A\sim \{[\mb{z}]\in Q_{1,1}:z_3-z_4=0\}$.
            \end{enumerate}
          \item {\bf Non-totally real type:} $X$ and $Y$ are linearly independent, which span a 2-plane $P_A$.
			\begin{enumerate}
			\item [(a)]{\bf Hyperbolic type:} $P_A$ is space-like, i.e. $\lan X\w Y,X\w Y\ran>0$. In this case, there exists
			a unique $u\in (0,+\infty]$, such that $[A]\sim [(\tanh u,i,0,0)]$ (here $\tanh (+\infty):=1$), and $H_A\sim \{[\mb{z}]\in Q_{1,1}:\tanh u\ z_1+iz_2=0\}$.
			\item [(b)] {\bf Elliptic type:} $P_A$ is time-like, i.e. $\lan X\w Y,X\w Y\ran<0$. In this case, there exists
			a unique $\a\in (0,\f{\pi}{2})$, such that $[A]\sim [(0,0,1,i\tan \a)]$, and $H_A\sim \{[\mb{z}]\in Q_{1,1}:z_3-i\tan \a z_4=0\}$.
			\item [(c)] {\bf Parabolic type:} $P_A$ is light-like, i.e. $\lan X\w Y,X\w Y\ran=0$. In this case, $[A]\sim [(1,0,i,i)]$,
and $H_A\sim \{[\mb{z}]\in Q_{1,1}: z_1+iz_3-iz_4=0\}$.
		\end{enumerate}
		
		\end{itemize}
	\end{thm}

	\begin{proof}
The proof for totally-real case is easy and direct, so we only consider the non-totally real case.

     Similarly as in \S 2 of \cite{H-O}, we can choose $\th\in [0,2\pi)$, such that the real part $X'$ and the imaginary part $Y'$ of $A':=e^{i\th}A$ satisfy
      $\lan X',Y'\ran=0$. So we can assume $\lan X,Y\ran=0$ without loss of generality.
		
		If $\lan X\w Y,X\w Y\ran>0$, then both $X$ and $Y$ are space-like vectors. Without loss of generality we can assume
		$\lan X,X\ran\in (0,1]$ and $\lan Y,Y\ran=1$. Hence we can choose $\mb{e}_1,\mb{e}_2\in \R^{3,1}$, so that $\lan \mb{e}_1,\mb{e}_1\ran=1$, $X=\tanh u\ \mb{e}_1$ with $u\in (0,\infty]$
		and $Y=\mb{e}_2$. This extends to an oriented orthonormal basis $\{\mb{e}_1,\mb{e}_2,\mb{e}_3,\mb{e}_4\}$ of $\R^{3,1}$, such that the fourth
		component of $\mb{e}_4$ is positive. Let $\si\in SO^+(3,1)$, such that $\si(\mb{e}_i)=\ep_i$ for $i=1,2,3,4$, then $[\si(A)]=[(\tanh u,i,0,0)]$.
		
		If $\lan X\w Y,X\w Y\ran<0$, without loss of generality we can assume $\lan X,X\ran=1$, $\lan Y,Y\ran<0$ and $Y_4>0$.
		Therefore, we can take $\mb{e}_3,\mb{e}_4\in \R^{3,1}$, such that $X=\mb{e}_3$, $\lan\mb{e}_4,\mb{e}_4\ran=-1$ and $Y=\tan\a\mb{e}_4$ with $\a\in (0,\f{\pi}{2})$.
		This extends to an oriented orthonormal basis $\{\mb{e}_1,\mb{e}_2,\mb{e}_3,\mb{e}_4\}$ of $\R^{3,1}$, which yields $\si\in SO^+(3,1)$
		satisfying $[\si(A)]=[(0,0,1,i\tan\a)]$.
		
		If $\lan X\w Y,X\w Y\ran=0$, without loss of generality we can assume $\lan X,X\ran=1$ and $Y$ is a null vector with $Y_4>0$. Select
		an oriented orthonormal basis $\{\mb{e}_1,\mb{e}_2,\mb{e}_3,\mb{e}_4\}$ of $\R^{3,1}$, such that $X=\mb{e}_1$, $(\mb{e}_2)_4=(\mb{e}_3)_4=0$ and
		$(\mb{e}_4)_4>0$. It follows from $X\perp Y$ and $Y_4>0$ that $Y=y_2 \mb{e}_2+y_3\mb{e}_3+y_4\mb{e}_4$ with $y_4>0$. Making a suitable rotation to $\{\mb{e}_2,\mb{e}_3\}$ we get $\{\mb{e}'_2,\mb{e}'_3\}$,
		so that $Y=y_4 \mb{e}'_3+y_4\mb{e}_4$. Let $u:=\log y_4$,
		$$\mb{e}''_3:=\cosh u\ \mb{e}'_3+\sinh u\ \mb{e}_4,\quad \mb{e}'_4:=\sinh u\ \mb{e}'_3+\cosh u\ \mb{e}_4,$$
		then $\{\mb{e}_1,\mb{e}'_2,\mb{e}''_3,\mb{e}'_4\}$ is an oriented orthonormal basis of $\R^{3,1}$ and $A=\mb{e}_1+i(\mb{e}''_3+\mb{e}'_4)$.
		Similarly to above, we concludes that $[A]$ is equivalent to $[(1,0,i,i)]$ under $SO^+(3,1)$-action.
		
		Obviously
		\begin{equation}
		I([A]):=\f{|\lan A,A\ran|}{\lan A,\bar{A}\ran}
		\end{equation}
		is $O(3,1)$-invariant. A direct computation shows:
		\begin{itemize}
			\item If $[A]$ is hyperbolic, $I([A])=\f{1-\tanh^2 u}{1+\tanh^2 u}=\text{sech}(2u)\in [0,1)$. (Here $\text{sech}(+\infty)=0$.)
			\item If $[A]$ is elliptic, $I([A])=\f{1+\tan^2\a}{1-\tan^2\a}=\sec(2\a)\in \{\la\in \R:|\la|>1\}\cup \{\infty\}$.
			\item If $[A]$ is parabolic, $I([A])=1$.
		\end{itemize}
Thus, any 2 elements $[A],[B]$ lying in distinct orbits of $SO^+(3,1)$ cannot be equivalent under the action of $O(3,1)$. Furthermore,
 $SO^+(3,1)\subset SO(3,1)\subset O(3,1)$ implies the orbits of the 3 groups going through $[A]$ should coincide.
		
	\end{proof}
	
	
		

	\subsection{The conjugate similarity on $SL(2,\C)$}\label{conj}
	
	A natural problem comes from Proposition \ref{hp2}: {\it Given 2 hyperplanes $H_1,H_2\subset Q_{1,1}$, which conditions on $\Psi(H_1),\Psi(H_2)\subset \C^*\times
	\C^*$ ensures $H_1$ and $H_2$ are equivalent under the action of $SO^+(3,1)$?} We shall study the classification of hyperplanes in $Q_{1,1}$ from this viewpoint.
	
	The projection $\pi:\mb{u}\in \R^{3,1}\backslash \{0\}\mapsto [\mb{u}]\in \Bbb{PR}^{3,1}$ restricted on the light cone $\mc{L}$
	defines a $\R^\times$-bundle from $\mc{L}$ onto {\it the M\"{o}bius space} M\"{o}b, where $\R^\times$ is the multiplicative group
	of nonzero real numbers. The conformal structure on M\"{o}b is induced by the Minkowski inner product on $\R^{3,1}$, and
	$f:\mb{x}\in \Bbb{S}^2\mapsto [(\mb{x},1)]\in $M\"{o}b is a conformal diffeomorphism (see e.g \S 1 of \cite{H}). Noting the action $\si([\mb{u}])=[\si(\mb{u})]$
	for $\si\in SO^+(3,1)$
	sends M\"{o}b onto itself, $\chi:\si\mapsto f^{-1}\circ \si\circ f$ is an injective homomorphism from $SO^+(3,1)$ into $\text{Conf}(\Bbb{S}^2)$,
	the Lie group of all conformal diffeomorphisms of $\Bbb{S}^2$ onto itself. $\dim SO^+(3,1)=\dim \text{Conf}(\Bbb{S}^2)=6$ and the connectivity of
	$SO^+(3,1)$ implies $\chi$ gives an isomorphism between $SO^+(3,1)$ and the M\"{o}bius transformation group on $\Bbb{S}^2$. Due to the geometric meaning
	of $\Psi:Q_{1,1}\ra \C^*\times \C^*$ (see (\ref{geo_m}) and its context), the following diagram commutes:
	\begin{equation}\label{mob1}\CD
	Q_{1,1} @>\si>> Q_{1,1}  \\
	@V\Psi VV     @VV\Psi V \\
	\C^*\times \C^*  @>\psi>> \C^*\times \C^*
	\endCD
	\end{equation}
	Here $\si\in SO^+(3,1)$ and
	\begin{equation}\label{action1}
	\psi:(w_1,w_2)\mapsto (\f{aw_1+b}{cw_1+d},\f{\bar{a}w_2+\bar{b}}{\bar{c}w_2+\bar{d}})
	\end{equation}
is induced by the M\"{o}bius transformation $\chi(\si):w\mapsto \f{aw+b}{cw+d}$,
 whose coefficient matrix is
$T:=\begin{pmatrix} a & b\\ c & d\end{pmatrix}$.
Without loss of generality, we can assume $T\in SL(2,\C)$, and denote by $\mc{M}_T$ the corresponding M\"{o}bius transformation,
then $\mc{M}_T=\mc{M}_{T'}$ if and only if $T'=\pm T$.
	
	Given $H_A\subset Q_{1,1}$ with $[A]\in Q_{1,1}$ and $H_B=\si(H_A)$ with $\si\in SO^+(3,1)$, then Proposition \ref{hp2} tells us
	\begin{equation}
\Psi(H_A)=\{(w_1,w_2):w_1=c_1\text{ or }w_2=c_2\}
\end{equation}
	and $c_2\neq \bar{c}_1$ if and only if $[A]\in Q_{1,1}^+$; By (\ref{mob1}) and (\ref{action1}),
	there exists a M\"{o}bius transformation $\mc{M}_T$, such that
	\begin{equation}
\Psi(H_B)=\{(w_1,w_2):w_1=\mc{M}_T(c_1)\text{ or }w_2=\mc{M}_{\bar{T}}(c_2)\}.
\end{equation}
	Here $\mc{M}_{\bar{T}}(c_2)=\overline{\mc{M}_T(c_1)}$ if and only if $c_2=\bar{c}_1$.
Since the $SL(2,\C)$-action $T\cdot (c_1,c_2):=(\mc{M}_T(c_1),\mc{M}_{\bar{T}}(c_2))$ is transitive on
both $\{(c_1,c_2):c_2=\bar{c}_1\}$ and $\{(c_1,c_2):c_2\neq \bar{c}_1\}$, $\{H_A:[A]\in Q_{1,1}^+\}$ and $\{H_B:[B]\in Q_{1,1}\backslash Q_{1,1}^+\}$ are both orbits of $SO^+(3,1)$ acting on the set of hyperplanes in $Q_{1,1}$.	

	For $H_1:=H_A$ and $H_2:=H_B$ with $[A],[B]\notin Q_{1,1}$, Proposition \ref{hp2} ensures the existence of
	$S_1,S_2\in SL(2,\C)$, such that
	\begin{equation}
	\Psi(H_i)=\{(w,\mc{M}_{S_i}(w)):w\in \C^*\}\qquad \forall i\in \{1,2\}.
	\end{equation}
	If $H_1\sim H_2$, there exists $T\in SL(2,\C)$, such that
	\begin{equation}\label{action}\aligned
	\Psi(H_2)&=\{(w,\mc{M}_{S_2}(w)):w\in \C^*\}\\
	&=\{(\mc{M}_T(w),\mc{M}_{\bar{T}}\circ \mc{M}_{S_1}(w)):w\in \C^*\}\\
	&=\{(w',\mc{M}_{\bar{T}}\circ \mc{M}_{S_1}\circ \mc{M}_{T^{-1}}(w')):w'\in \C^*\}\\
	&=\{(w,\mc{M}_{\bar{T}S_1T^{-1}}(w)):w\in \C^*\}.
	\endaligned
    \end{equation}
Hence $S_2=\pm \bar{T}S_1T^{-1}$. Conversely, if $S_2=\pm \bar{T}S_1T^{-1}$ with $T\in SL(2,\C)$, it is easy to get
	$H_1\sim H_2$. This leads to the following conclusion:
	
	\begin{pro}\label{conj3}
		Let $H_1,H_2$ be hyperplanes in $Q_{1,1}$, such that $\Psi(H_1)$ and $\Psi(H_2)$ are respectively the graph of the M\"{o}bius transformation $\mc{M}_{S_1}$
		and $\mc{M}_{S_2}$, with $S_1,S_2\in SL(2,\C)$. Then $H_1\sim H_2$ if and only if there exists $T\in SL(2,\C)$, such that
		$S_2=\pm \bar{T}S_1T^{-1}$. In this case, $S_1$ and $S_2$ are said to be {\bf conjugate similar} to each other, and we denote
		$S_1\stackrel{conj}{\sim}S_2$.
	\end{pro}
	
	It is easily-seen that the conjugate similarity is an equivalence relation on $SL(2,\C)$. We shall explore the necessary and sufficient conditions
	ensuring 2 matrices to be conjugate similar, and examine the conjugate similar canonical forms.
	
	For $S\in SL(2,\C)$, define
	\begin{equation}\label{RS1}
	\mc{R}_S(v)=\overline{Sv},
	\end{equation}
	then $\mc{R}_S$ is obviously a $\R$-linear mapping. Let $A$ and $B$ be the real and imaginary part of $S$, respectively, then
	$$\mc{R}_S(x+iy)=\overline{(A+iB)(x+iy)}=Ax-By-i(Bx+Ay).$$
	This means
	\begin{equation}
	R_S:=\begin{pmatrix}
	A & -B\\
	-B & -A
	\end{pmatrix}
	\end{equation}
	is the representative matrix of $\mc{R}_S$.
	
A nonnegative real number $r$ is called {\it conjugate eigenvalue} of $S$, whenever there exists $v\neq 0$,
such that
\begin{equation}
Sv=r\bar{v},
\end{equation}
and	$v$ is called a {\it conjugate eigenvector} of $S$ associated to $r$. All such vectors
form a $\R$-linear subspace of $\C^2$ and its dimension is said to be the {\it multiplicity} of $r$.
Due to the definition of conjugate similarity and $R_S$, we have:
\begin{itemize}
\item Conjugate similar matrices $S_1,S_2$ share the same conjugate eigenvectors;
\item $r\in \R^+$ is a conjugate eigenvalue of $S$ if and only if $r$
	is an eigenvalue of $R_S$;
\item $v=x+iy$ is a conjugate eigenvector of $S$
	if and only if $\begin{pmatrix}
	x \\
	y
	\end{pmatrix}$ is an eigenvector of $R_S$.
\end{itemize}
	
	By (\ref{RS1}),
	$$\mc{R}_S(iv)=\overline{S(iv)}=-i\overline{Sv}=-i\mc{R}_S(v)$$
	and hence
	\begin{equation}\label{RS2}
	R_S J=-JR_S\qquad \text{with }J=\begin{pmatrix}
	0 & -I_2\\
	I_2 & 0
	\end{pmatrix}.
	\end{equation}
	This enables us to get:
	
	\begin{lem}\label{eig1}
		For a complex column vector $z$ and a complex number $\la$, $z$ is a (generalized) eigenvector of $R_S$ associated to the eigenvalue $\la$ if and only if
		$Jz$ is a (generalized) eigenvector of $R_S$ associated to the eigenvalue $-\la$.
	\end{lem}
	\begin{proof}
		If $z$ is an eigenvector of $R_S$ associated to $\la$, then $R_S z=\la z$ and (\ref{RS2}) gives
		$$R_S(Jz)=-J(R_S(z))=-\la (Jz),$$
		which means $Jz$ is an eigenvalue of $R_S$ associated to $-\la$. And verse visa.
		
		If $z$ is a generalized eigenvector of $R_S$ associated to $\la$, then there exists $k\in \Bbb{Z}^+$, such that
		$(R_S-\la I_4)^k(z)=0$, hence
		$$\aligned
		0&=J(R_S-\la I_4)^k z=(-R_S-\la I_4)J(R_S-\la I_4)^{k-1}z\\
		&=\cdots=(-R_S-\la I_4)^k (Jz).
		\endaligned$$
		In other words, $Jz$ is a generalized eigenvector of $R_S$ associated to $-\la$. And verse visa.
	\end{proof}
	
	Via carefully examining the eigenvalues and the Jordan canonical form of $R_S$, we can get all equivalence classes of conjugate similarity as below.
	
	\begin{thm}\label{class2}
		For each $S_1,S_2\in SL(2,\C)$, $S_2$ is conjugate similar to $S_1$ if and only if $R_{S_2}$ is similar to $R_{S_1}$.
		An arbitrary $S\in SL(2,\C)$ can be classified to one and only one of the following 3 types:
		\begin{enumerate}
			\item [(1)] $S\stackrel{conj}{\sim} \begin{pmatrix}
			e^u & 0\\
			0 & e^{-u}
			\end{pmatrix}$ with $u\geq 0$.
			\item [(2)] $S\stackrel{conj}{\sim} \begin{pmatrix}
			\cos\a & \sin\a\\
			-\sin\a & \cos\a
			\end{pmatrix}\stackrel{conj}{\sim} \begin{pmatrix}
			0 & ie^{-i\a}\\
			ie^{i\a} & 0
			\end{pmatrix}$ with $\a\in (0,\f{\pi}{2}]$.
			\item [(3)]
			$S\stackrel{conj}{\sim} \begin{pmatrix}
			1 & 1\\
			0 & 1
			\end{pmatrix}$.
		\end{enumerate}
	\end{thm}
	
	\begin{proof}
		If $S_1\stackrel{conj}{\sim} S_2$, then there exists $T\in SL(2,\C)$, such that $S_2=\pm \bar{T}S_1T^{-1}$.
		Lemma \ref{eig1} tells us $R_{S_2}$ is similar to $R_{-S_2}=-R_{S_2}$, hence we can assume $S_2=\bar{T}S_1T^{-1}$
		without loss of generality. Since
		$$\mc{R}_{S_2}(v)=\overline{S_2v}=\overline{\bar{T}S_1T^{-1}v}=T\mc{R}_{S_1}(T^{-1}(v)),$$
		we have
		$$R_{S_2}=\begin{pmatrix}
		X & -Y\\
		Y & X
		\end{pmatrix}R_{S_1}\begin{pmatrix}
		X & -Y\\
		Y & X
		\end{pmatrix}^{-1}$$
		with $X:=\text{Re }T$ and $Y:=\text{Im }T$.
		This implies $R_{S_1}\sim R_{S_2}$.
		
		For $S\in SL(2,\C)$, we shall consider the eigenvalues of $R_S$ case by case:
		\begin{itemize}
			\item {\bf Case I:} $R_S$ has a non-real eigenvalue $\la$.

Observing that $\bar{\la}$, $-\la$, $-\bar{\la}$ are also
			eigenvalues of $R_S$, we can assume $\la=\la_1+i\la_2$ with $\la_1\geq 0$ and $\la_2>0$, without loss of generality.
			Let $z:=\begin{pmatrix}
			x_1 \\
			y_1
			\end{pmatrix}+\sqrt{-1}\begin{pmatrix}
			x_2 \\
			y_2
			\end{pmatrix}$ be the eigenvector associated to $\la$, then $R_S z=\la z$ gives
			\begin{equation}\left\{\begin{array}{c}
			\mc{R}_S(v_1)=\la_1 v_1-\la_2 v_2\\
			\mc{R}_S(v_2)=\la_2 v_1+\la_1 v_2
			\end{array}\right.\qquad \text{with }v_1:=x_1+iy_1,v_2:=x_2+iy_2
			\end{equation}
			and moreover
			\begin{equation}
			S(v_1\ v_2)=(\bar{v}_1\ \bar{v}_2)\begin{pmatrix}
			\la_1 & \la_2\\
			-\la_2 & \la_1
			\end{pmatrix}.
			\end{equation}
			The assumption that $v_1$ and $v_2$ are $\C$-linear dependent forces $\la_2=0$, causing a contradiction. Therefore
			$T_0:=(v_1\ v_2)$ is non-singular and
			\begin{equation}\label{conj1}
			S=\bar{T}_0\begin{pmatrix}
			\la_1 & \la_2\\
			-\la_2 & \la_1
			\end{pmatrix}T_0^{-1}.
			\end{equation}
			Taking the determinants of both sides of the above equality yields $\overline{|T_0|}(\la_1^2+\la_2^2)|T_0|^{-1}=1$, hence
			$|T_0|\in \R$ and $\la_1^2+\la_2^2=1$. If $|T_0|>0$, taking $T:=\mu T_0$ with a suitable real number $\mu$ makes sure $|T|=1$ and
			$S=\bar{T}\begin{pmatrix}
			\la_1 & \la_2\\
			-\la_2 & \la_1
			\end{pmatrix}T^{-1}$. If $|T_0|<0$, taking $T:=\nu iT_0$ with a suitable real number $\nu$ makes sure $|T|=1$ and
			$S=-\bar{T}\begin{pmatrix}
			\la_1 & \la_2\\
			-\la_2 & \la_1
			\end{pmatrix}T^{-1}$.  Hence $S\stackrel{conj}{\sim} \begin{pmatrix}
			\cos\a & \sin\a\\
			-\sin\a & \cos\a
			\end{pmatrix}$ with $\a\in (0,\f{\pi}{2}]$, and $R_S\sim \text{diag}(e^{i\a},e^{-i\a},-e^{i\a},-e^{-i\a})$.
			Therefore, given $S_1,S_2$ belonging to this case, $S_1\stackrel{conj}{\sim} S_2$ if and only if $R_{S_1}\sim R_{S_2}$.
			In particular, if $S=\begin{pmatrix}
			0 & ie^{-i\a}\\
			ie^{i\a} & 0
			\end{pmatrix}$, the easily-seen fact that $e^{i\a}$ is an eigenvalue of $R_S$ gives
			$\begin{pmatrix}
			0 & ie^{-i\a}\\
			ie^{i\a} & 0
			\end{pmatrix}\stackrel{conj}{\sim} \begin{pmatrix}
			\cos\a & \sin\a\\
			-\sin\a & \cos\a
			\end{pmatrix}$.
			
			\item {\bf Case II:} $R_S$ is a diagonalizable matrix whose eigenvalues are all real.

By Lemma \ref{eig1},
			there exist $r_1\geq r_2\geq 0$ and $u_1=\begin{pmatrix}
			x_1 \\
			y_1
			\end{pmatrix},u_2=\begin{pmatrix}
			x_2 \\
			y_2
			\end{pmatrix}$, such that $u_1,Ju_1,u_2,Ju_2$ are $\R$-linear independent, and
			\begin{equation}\label{eig2}
			R_Su_k=r_k u_k,\quad R_S(Ju_k)=-r_k Ju_k,\qquad \forall k=1,2.
			\end{equation}
			Thus $v_1:=x_1+iy_1$ and $v_2:=x_2+iy_2$ are $\C$-linear independent, and $\mc{R}_S(v_1)=r_1 v_1$,
			$\mc{R}_S(v_2)=r_2 v_2$. In other words, $T:=(v_1\ v_2)$ satisfies
			\begin{equation}
			S=\bar{T}\begin{pmatrix}
		r_1 & 0\\0 & r_2
			\end{pmatrix}T^{-1}.
			\end{equation}
			Taking the determinants of the both sides gives $r_1 r_2=1$ and $|T|\in \R$. Without loss of generality we can assume $|T|=1$,
			and hence $S\stackrel{conj}{\sim} \begin{pmatrix}
			e^u & 0\\
			0 & e^{-u}
			\end{pmatrix}$ with $u\geq 0$. In conjunction with (\ref{eig2}),
			$S_1\stackrel{conj}{\sim} S_2$ if and only if $R_{S_1}\sim R_{S_2}$ for $S_1,S_2$ belonging to this case.
			
			\item {\bf Case III:} $R_S$ is a non-diagonalizable matrix whose eigenvalues are all real.

Due to Lemma \ref{eig1},
			there exists $r\geq 0$, such that
			\begin{equation}
			R_S\sim \begin{pmatrix}
			J_1 & \\
			& J_2
			\end{pmatrix}
			\qquad \text{with }J_1:=\begin{pmatrix}
			r & 1\\
			0 & r
			\end{pmatrix}, J_2:=\begin{pmatrix}
			-r & 1\\
			0 & -r
			\end{pmatrix}.
			\end{equation}
			This means the existence of $u_1=\begin{pmatrix}
			x_1 \\
			y_1
			\end{pmatrix},u_2=\begin{pmatrix}
			x_2 \\
			y_2
			\end{pmatrix}$, such that $u_1,Ju_1,u_2,Ju_2$ are $\R$-linear independent, and
			\begin{equation}
			R_Su_1=r u_1,\quad R_Su_2=r u_2+u_1.
			\end{equation}
			Similarly to above, we can deduce that
			\begin{equation}
			S=\bar{T}\begin{pmatrix}
			r & 1\\
			0 & r
			\end{pmatrix}T^{-1}
			\end{equation}
			with $T:=(x_1+iy_1\ x_2+iy_2)$ being a non-singular matrix. Again taking the determinants of the both sides gives $|T|\in \R$
			and $r=1$. We can assume $|T|=1$ without loss of generality, and hence $S\stackrel{conj}{\sim} \begin{pmatrix}
			1 & 1\\
			0 & 1
			\end{pmatrix}$.
Therefore, $S_1,S_2$ belonging to this case are conjugate similar to each other, and $R_{S_1}\sim R_{S_2}$.

		\end{itemize}

	\end{proof}

Conjugate eigenvectors play a crucial part in the study of the structure of hyperplanes in $Q_{1,1}^+$, as shown in the
following proposition:
	
	\begin{pro}\label{conj2}
		Let $H$ be a hyperplane in $Q_{1,1}$, such that $\Psi(H)=\{(w,\mc{M}_S(w)):w\in \C^*\}$ with $S\in SL(2,\C)$, and $H^+:=H\cap Q_{1,1}^+$,
 then
$\Psi(H^+)$ is just the graph of $\mc{M}_S$ over $\C^*\backslash E_S$, where
\begin{equation}
E_S= \left\{w=\f{v_1}{v_2}: \begin{pmatrix} v_1\\ v_2\end{pmatrix}\text{ is a conjugate eigenvector of } S\right\}.
\end{equation}
	\end{pro}

\begin{proof}
By Proposition \ref{com_str},
	\begin{equation}
	\Psi(H^+)=\{(w,\mc{M}_S(w)):\mc{M}_S(w)\neq \bar{w}\}.
	\end{equation}
So it suffices to clarify all solutions of $\bar{w}=\mc{M}_S(w)$.

If $\bar{w}=\mc{M}_S(w)$, then for $(z_1,z_2)\neq (0,0)$ satisfying
	$\f{z_1}{z_2}=w$ and $z:=\begin{pmatrix} z_1\\ z_2\end{pmatrix}$, we have $Sz=\la \bar{z}$ with $\la$ a non-zero complex number. Denote
	$\la=r e^{2i\th}$ with $r=|\la|> 0$, then putting $v:=e^{-i\th}z$ gives $Sv=r \bar{v}$, i.e. $v$ is a conjugate eigenvector of $S$.
	Conversely, a conjugate eigenvector $\begin{pmatrix} v_1\\ v_2\end{pmatrix}$ of $S$ immediately forces $\bar{w}=\mc{M}_S(w)$ with
	$w:=\f{v_1}{v_2}$. This completes the proof of the present proposition.
\end{proof}

	In conjunction with (\ref{inv_F}), Theorem \ref{class1}, Theorem \ref{class2} and Proposition \ref{conj2}, we obtain the classification of all hyperplanes
	in $Q_{1,1}$ (or $Q_{1,1}^+$) as follows:
	\begin{thm}\label{class3}
	Each hyperplane $H_A\subset Q_{1,1}$ with $[A]\in \Bbb{CP}^{2,1}$ can be categorized in one and only one of the following classes:
		\begin{itemize}
			\item  $[A]\in Q_{1,1}$, i.e. $\Psi(H_A)=\{(w_1,w_2)\in \C^*\times \C^*:w_1=c_1\text{ or }w_2=c_2\}$:
			\begin{enumerate}
				\item [(a)] $[A]\in Q_{1,1}^+$ if and only if $c_2\neq \bar{c}_1$. In this case, $[A]\sim [(1,i,0,0)]$.
				\item [(b)] $[A]\notin Q_{1,1}^+$ if and only if $c_2=\bar{c}_1$. In this case, $[A]\sim [\mb{u}_N]$,
			\end{enumerate}
			
			\item $[A]\notin Q_{1,1}$, i.e. $\Psi(H_A)=\{(w,\mc{M}_S(w)):w\in \C^*\}$ with $S\in SL(2,\C)$:
			\begin{enumerate}
				\item [(a)] $S$ has infinite many conjugate eigenvalues if and only if $[A]\sim [\mb{u}_S]$.
  In this case, $S\stackrel{conj}{\sim} \begin{pmatrix}
				1 & 0\\
				0 & 1
				\end{pmatrix}$ and $H_A^+$ is conformally equivalent to $\C\backslash \R$.
				
				
				\item [(b)] $S$ has 2 distinct conjugate eigenvalues if and only if $H_A$ is a hyperplane of hyperbolic type.
				In this case, $S\stackrel{conj}{\sim} \begin{pmatrix}
				e^u & 0\\
				0 & e^{-u}
				\end{pmatrix}$ if and only if $[A]\sim [(\sinh u,i\cosh u,0,0)]$, where $u\in (0,+\infty)$, and $H_A^+$ is conformally equivalent to $\C\backslash \{0\}$.
				
				\item [(c)] $S$ has no conjugate eigenvalue if and only if $[A]\sim [\mb{u}_T]$ (totally real case) or $H_A$ is a hyperplane of elliptic type (non-totally real case).
				In this case, $S\stackrel{conj}{\sim} \begin{pmatrix}
				0 & ie^{-i\a}\\
				ie^{i\a} & 0
				\end{pmatrix}$ if and only if $[A]\sim [(0,0,\cos\a,i\sin\a)]$, where $\a\in (0,\f{\pi}{2}]$, and $H_A^+$ is conformally equivalent to $\C^*$.

				\item [(d)] $S$ has exactly $1$ conjugate eigenvalue if and only if $H_A$ is a hyperplane of parabolic type. In this case,
				$S\stackrel{conj}{\sim} \begin{pmatrix}
				1 & 1\\
				0 & 1
				\end{pmatrix}$, $[A]\sim [(1,0,i,i)]$ and $H_A^+$ is conformally equivalent to $\C$.
				
			\end{enumerate}
			
		\end{itemize}
		
	\end{thm}
	
	\subsection{Metrics of hyperplanes in $Q_{1,1}^+$}\label{metric4}
	
	The metric of an arbitrary hyperplane $H^+:=H\cap Q_{1,1}^+$ is induced by the canonical metric of
	$Q_{1,1}^+$, which is invariant under $SO^+(3,1)$-action. Hence each $\si\in SO^+(3,1)$ which maps
	$H$ onto itself yields an isometry of $H^+$. Let $\mc{M}_T$ ($T\in SL(2,\C)$) be the M\"{o}bius transformation
corresponding to $\si$; if $\Psi(H)$ is the graph of $\mc{M}_S$, then (\ref{mob1}) and (\ref{action1}) enable us to conclude that
	$\Psi(\si(H))$ is the graph of $\mc{M}_{\bar{T}ST^{-1}}$, as in (\ref{action}).
 Hence $\si(H)=H$ if and only if $\bar{T}ST^{-1}=\pm S$. Namely,
	\begin{equation}
	G_S:=\{T\in SL(2,\C):\bar{T}ST^{-1}=\pm S\}
	\end{equation}
	is a Lie group, such that the $G_S$-action
	\begin{equation}
	T\cdot \Psi^{-1}(w,\mc{M}_S(w))=\Psi^{-1}(\mc{M}_T(w),\mc{M}_{\bar{T}S}(w))\qquad \forall w\in \C^*\backslash E_S
	\end{equation}
	on $H^+$ keeps its metric invariant. Let $\textswab{g}_S$ be the Lie algebra of $G_S$, then
	\begin{equation}
	\textswab{g}_S=\{X\in sl(2,\C):\bar{X}S-SX=0\}.
	\end{equation}

	Note that any 2 hyperplanes $H_1^+,H_2^+\subset Q_{1,1}^+$ lying in the same equivalence class of $SO^+(3,1)$-action are isometric to each other, it suffices to consider the metrics of representative hyperplanes as follows:
	
	\begin{itemize}
		\item {\bf Case I.} $S=\begin{pmatrix}
		1 & 0\\
		0 & 1
		\end{pmatrix}$. Then a direct calculation shows $\textswab{g}_S=sl(2,\R)$ and $G_S=SL(2,\R)\cup \text{diag}(i,-i)SL(2,\R)$, which is a $3$-dimension Lie group
		acting transitively on $H^+$. In fact, substituting $w_2=w_1=w$ into (\ref{g2}) gives
		\begin{equation}
		g=-\f{|dw|^2}{(\text{Im }w)^2}\qquad w\in \C\backslash \R.
		\end{equation}
		Therefore, $(H^+,-g)$ has 2 connected component and each component is isometric to the complete hyperbolic plane with the constant Gauss curvature $-1$.
		
		\item {\bf Case II.} $S=\begin{pmatrix}
		0 & 1\\
		-1 & 0
		\end{pmatrix}$. Then $\textswab{g}_S=su(2,\C)$ and $G_S=SU(2,\C)$, which acts transitively on $H^+$.
		In fact, substituting $w_1=w$ and $w_2=-\f{1}{w}$ into (\ref{g2}) gives
		\begin{equation}
		g=\f{4|dw|^2}{(1+|w|^2)^2}
		\end{equation}
		Therefore, $(H^+,g)$ is isometric to the unit sphere equipped with the canonical metric.
		
		\item {\bf Case III.} $S=\begin{pmatrix}
		e^u & 0\\
		0 & e^{-u}
		\end{pmatrix}$ with $u\in (0,+\infty)$. Then
\begin{equation}\textswab{g}_S=\R \begin{pmatrix}
		1 & 0\\
		0 & -1
		\end{pmatrix}, G_S=\left\{\begin{pmatrix}
		\la & 0\\
		0 & \la^{-1}
		\end{pmatrix}:\la\in \R\text{ or }i\R,\la\neq 0 \right\}.
\end{equation}
		Therefore, $H^+$ is diffeomorphic to $\C\backslash \{0\}$ and the metric $g$ is invariant under the scaling
		$z\mapsto kz$ for each $k\in \R\backslash \{0\}$. In fact, substituting $w_1=w$ and $w_2=e^{2u}w$ into (\ref{g2})
		and then letting $w=e^{t+i\th}$ with $t\in \R, \th\in \R/(2\pi\Bbb{Z})$ implies
		\begin{equation}
		g=\text{Re}\left[\f{4}{(e^{-u}e^{-i\th}-e^{u}e^{i\th})^2}\right](dt^2+d\th^2).
		\end{equation}
Let $dA_g$ be the area form associated to $g$, then $dA_g$ is also
		invariant under the scaling and hence $\int_{H^+} dA_g$ is divergent.
		
		\item {\bf Case IV.} $S=\begin{pmatrix}
		0 & ie^{-i\a}\\
		ie^{i\a} & 0
		\end{pmatrix}$ with $\a\in (0,\f{\pi}{2})$. Then
\begin{equation}\textswab{g}_S=\R \begin{pmatrix}
		i & 0\\
		0 & -i
		\end{pmatrix}, G_S=\left\{\begin{pmatrix}
		e^{it} & 0\\
		0 & e^{-it}
		\end{pmatrix}:t\in \R/(2\pi\Bbb{Z}) \right\}.
\end{equation}
		Therefore, $H^+$ is diffeomorphic to $\C^*$ and the metric $g$ is invariant under the rotation
		$z\mapsto e^{i\be}z$ for each $\be\in \R/(2\pi\Bbb{Z})$. In fact, substituting $w_1=w$ and $w_2=e^{-2i\a}w^{-1}$ into (\ref{g2})
		and then letting $w=e^{t+i\th}$ with $t\in \R$, $\th\in \R/(2\pi\Bbb{Z})$ implies
		\begin{equation}
		g=-\text{Re}\left[\f{4}{(e^{-t} e^{-i\a}-e^{t}e^{i\a})^2}\right](dt^2+d\th^2).
		\end{equation}
A direct calculation shows
		\begin{equation}\label{area}
		\int_{H^+}dA_g=4\pi.
		\end{equation}
		
		\item {\bf Case V.} $S=\begin{pmatrix}
		1 & 1\\
		0 & 1
		\end{pmatrix}$. Then
\begin{equation}\textswab{g}_S=\R \begin{pmatrix}
		0 & 1\\
		0 & 0
		\end{pmatrix}, G_S=\left\{\pm\begin{pmatrix}
		1 & t\\
		0 & 1
		\end{pmatrix}:t\in \R \right\}.
\end{equation}
Therefore, $H^+$ is diffeomorphic to $\C$ and the metric $g$ is invariant under the parallel translation
		$z\mapsto z+t$ for each $t\in \R$. In fact, substituting $w_1=w$ and $w_2=w+1$ into (\ref{g2})
		and then letting $w=x+iy$ with $x,y\in \R$ implies
		\begin{equation}
		g=\text{Re}\left[\f{4}{(1+2iy)^2}\right](dx^2+dy^2).
		\end{equation}
Since $dA_g$ is also
		invariant under the parallel translation, $\int_{H^+} dA_g$ should be divergent.

	\end{itemize}

On the other hand, for $H^+:=H_A^+$ with $[A]\in Q_{1,1}$, $\Psi(H^+)$ lies in $(\{c_1\}\times \C^*)\cup (\C^*\times \{c_2\})$ with $c_1,c_2\in \C^*$.
	Taking $w_1\equiv c_1$ or $w_2\equiv c_2$ in (\ref{g2}) forces $g=0$, i.e. the induced metric on $H^+$ vanishes everywhere.

	\bigskip\bigskip
	
	\Section{The Value distribution for Gauss maps of complete stationary surfaces}{The Value distribution for Gauss maps of complete stationary surfaces}\label{S3}
	
	Let $\mb{x}:M\ra \R^{3,1}$ be an oriented surface in the Minkowski space. If the pull-back metric $ds^2:=\lan d\mb{x},d\mb{x}\ran$
	is positive definite everywhere, $M$ is called a {\it space-like surface}. Moreover, $M$ is said to be {\it stationary} whenever
	the mean curvature vector field $\mb{H}\equiv 0$. $M$ is stationary if and only if the restriction of each coordinate function
	on $M$ is harmonic.
	
	Via parallel translation, $G:p\in M\mapsto T_p M$ gives a smooth mapping from $M$ into the Lorentz Grassmannian manifold $\mb{G}_{1,1}^2$,
	which is called the {\it generalized Gauss map} of $M$.
	Let $(u,v)$ be local oriented isothermal parameters on a neighborhood of $p$, then $\lan \mb{x}_u,\mb{x}_u\ran=\lan \mb{x}_v,\mb{x}_v\ran>0$
	and $\lan \mb{x}_u,\mb{x}_v\ran=0$. Denote by $z:=u+iv$ the local complex coordinate
 of $M$, then $[\mb{x}_z]=[\f{1}{2}(\mb{x}_u-i\mb{x}_v)]$ is the point in $Q_{1,1}^+$ corresponding
	to the Gauss image of $p$, which is independent of the choice of isothermal coordinates. Denote
	\begin{equation}
	\phi=(\phi_1,\phi_2,\phi_3,\phi_4):=\Big(\pd{x_1}{z},\pd{x_2}{z},\pd{x_3}{z},\pd{x_4}{z}\Big)dz=\mb{x}_z dz
	\end{equation}
	then the harmonicity of $x_k$ forces $\phi_k$ to be a holomorphic 1-form that can be globally defined on $M$.
	$[\phi]=[\mb{x}_z]\in Q_{1,1}^+$ is equivalent to saying that
	\begin{eqnarray}
	\phi_1^2+\phi_2^2+\phi_3^2-\phi_4^2&=&0,\label{phi1}\\
	|\phi_1|^2+|\phi_2|^2+|\phi_3|^2-|\phi_4|^2&>&0\label{phi2}.
	\end{eqnarray}
	Conversely, if $\phi_1,\phi_2,\phi_3,\phi_4$ be holomorphic 1-forms on a Riemann surface $M$, satisfying
	(\ref{phi1}) and the period condition
	\begin{equation}\label{phi3}
	\text{Re}\oint_\g \phi_i=0\qquad \forall i\in \{1,2,3,4\}
	\end{equation}
	with $\g$ an arbitrary closed curve in $M$, then
	\begin{equation}\label{W6}
	\mb{x}:=\text{Re}\int (\phi_1,\phi_2,\phi_3,\phi_4)
	\end{equation}
	defines a {\it generalized stationary surface} in $\R^{3,1}$, equipped with the induced metric
 \begin{equation}
 ds^2=\lan \phi,\bar{\phi}\ran=|\phi_1|^2+|\phi_2|^2+|\phi_3|^2-|\phi_4|^2.
 \end{equation}
 Moreover, $ds^2$ is positive definite if and only if Condition (\ref{phi2}) holds.
 This is the {\it Weierstrass representation} for stationary
space-like surfaces.
	
Noting that $Q_{1,1}$ is conformally equivalent to $\C^*\times \C^*$ (see Proposition \ref{com_str}), we define
\begin{equation}
(\psi_1,\psi_2):=\Psi([\mb{x}_z])
\end{equation}
and the definition of $\Psi$ implies
\begin{itemize}
\item $\psi_1=\f{\phi_1+i\phi_2}{\phi_3+\phi_4}$, $\psi_2=\f{\phi_1-i\phi_2}{\phi_3+\phi_4}$ whenever $\phi_3+\phi_4\not\equiv 0$;
\item $\psi_1=\f{\phi_4-\phi_3}{\phi_1-i\phi_2}$, $\psi_2\equiv \infty$ whenever $\phi_3+\phi_4\equiv 0$, $\phi_1+i\phi_2\equiv 0$
and $\phi_1-i\phi_2\not\equiv 0$;
\item $\psi_1\equiv\infty$, $\psi_2=\f{\phi_4-\phi_3}{\phi_1+i\phi_2}$ whenever $\phi_3+\phi_4\equiv 0$, $\phi_1-i\phi_2\equiv 0$ and $\phi_1+i\phi_2\not\equiv 0$;
\item $\psi_1\equiv \infty$, $\psi_2\equiv \infty$ whenever $\phi_3+\phi_4\equiv 0$, $\phi_1-i\phi_2\equiv 0$ and $\phi_1+i\phi_2\equiv 0$.
\end{itemize}
Hence $\psi_1$ and $\psi_2$ are both meromorphic functions on $M$, which can be seen as the 2 components of the Gauss map of $M$.

As in \cite{M-W-W}, we let
\begin{equation}
dh:=\f{1}{2}(\phi_3+\phi_4)
\end{equation}
 be the {\it height differential}. if $dh \equiv 0$, then (\ref{phi1}) implies $\phi_1-i\phi_2\equiv 0$ or $\phi_1+i\phi_2\equiv 0$. Hence
\begin{equation}\label{phi6}
\phi=(\phi_1,\pm i\phi_1,\phi_3,-\phi_3)
\end{equation}
and $\phi_1$ has no zero.
	Otherwise, the zeros of $dh$ is discrete.
 Based on (\ref{phi1}), we can proceed as in the proof
of Proposition \ref{com_str} to get
	\begin{equation}\label{phi4}
	\phi=(\psi_1+\psi_2,-i(\psi_1-\psi_2),1-\psi_1\psi_2,1+\psi_1\psi_2)dh,
	\end{equation}
	\begin{equation}\label{phi5}
	ds^2=\lan \phi,\bar{\phi}\ran=2|\psi_1-\bar{\psi_2}|^2|dh|^2
	\end{equation}
	and the conditions (\ref{phi2})-(\ref{phi3}) are equivalent to the following constraints on the {\it W-data} $(\psi_1,\psi_2,dh)$  (see \cite{M-W-W}):
	\begin{itemize}
		\item $\psi_2\neq \bar{\psi}_1$ everywhere;
		\item  $p\in M$ is a zero of $dh$ if and only if $p$ is a pole of $\psi_i$ for a unique $i\in \{1,2\}$, with the same order;
		\item  $\oint_\g \psi_1 dh=-\overline{\oint_\g \psi_2 dh}$ and $\text{Re}\oint_\g dh=\text{Re}\oint_\g \psi_1\psi_2 dh=0$ for
		each closed path $\g$.
	\end{itemize}
	
	
	Combining with the Gauss equations and $\mb{H}\equiv 0$, we arrive at
	\begin{equation}
	G^* g=-Kds^2,
	\end{equation}
	where $g$ is the canonical metric on $Q_{1,1}^+$ and $K$ is the Gauss curvature. The metric expressions (\ref{g2}) and (\ref{phi5}) of $Q_{1,1}^+$
	and $M$ enable us to get
	\begin{equation}\label{K}
	K=-\text{Re}\left[\f{2\bar{\psi}'_1\psi'_2(\psi_1-\bar{\psi}_2)^2}{|\bar{\psi}_1-\psi_2|^6}\right]
	\end{equation}
	with $\psi'_i:=\f{d\psi_i}{dh}$ for $i\in \{1,2\}$.

Let $\pi:\td{M}\ra M$ be the universal covering map, $\td{\mb{x}}:=\mb{x}\circ \pi$, $\td{G}:=G\circ \pi$, $\td{\psi}_i:=\psi_i\circ \pi$
($i\in \{1,2\}$) be lifting of $\mb{x},G,\psi_i$, respectively, and $\td{\phi}:=\pi^* \phi$, $d\td{h}:=\pi^* dh$ be respectively
the pull-back of $\phi,dh$, then it is easy to verify that:
\begin{itemize}
\item $\td{\mb{x}}:\td{M}\ra \R^{3,1}$ gives a simply-connected space-like stationary surface.
\item All terms of $\td{\phi}$ are holomorphic 1-forms on $\td{M}$, and $\lan \td{\phi},\td{\phi}\ran=0$.
\item $\td{G}$ is just the Gauss map of $\td{M}$
and $(\td{\psi}_1,\td{\psi}_2,d\td{h})$ becomes the W-data of $\td{M}$. In particular, for $i=1$ or $2$, $\td{\psi}_i$ is a constant function if and only if $\psi_i$ is.
\item $\pi:(\td{M},d\td{s}^2)\ra (M,ds^2)$ is a local isometry, with $d\td{s}^2$ the induced metric of $\td{M}$. As
a corollary, $d\td{s}^2$ is a complete metric if and only if $ds^2$ is.
\end{itemize}

Assume $M$ is a simply connected Riemann surface and $\phi_i$ is a holomorphic $1$-form
on $M$ for each $1\leq i\leq 4$, satisfying Condition (\ref{phi1}). Then $\oint_\g \phi_i=0$ automatically holds
and $\mb{x}:M\ra \R^{3,1}$ defined by (\ref{W6}) gives a generalized stationary surface in $\R^{3,1}$. Letting
\begin{equation}
\phi^*=(\phi_1^*,\phi_2^*,\phi_3^*,\phi_4^*):=(\phi_1,\phi_2,\phi_3,i\phi_4)
\end{equation}
gives
\begin{equation}
(\phi_1^*)^2+(\phi_2^*)^2+(\phi_3^*)^2+(\phi_4^*)^2=0
\end{equation}
and hence
\begin{equation}
\mb{x}^*:=\text{Re}\int (\phi_1^*,\phi_2^*,\phi_3^*,\phi_4^*)
\end{equation}
defines a generalized minimal surface in $\R^4$ (see e.g \cite{C-O,H-O}), which is called the {\it dual immersion} of $\mb{x}$.
Obviously $\mb{x}\leftrightarrow \mb{x}^*$ gives a one-to-one correspondence between all simply-connected generalized stationary surfaces in $\R^{3,1}$
and all simply-connected generalized minimal surfaces in $\R^4$, which has the following properties:
\begin{itemize}
\item Let $(ds^*)^2$ and $ds^2$ be metrics induced by $\mb{x}^*$ and $\mb{x}$, respectively, then
\begin{equation}
(ds^*)^2=|\phi_1^*|^2+|\phi_2^*|^2+|\phi_3^*|^2+|\phi_4^*|^2\geq |\phi_1|^2+|\phi_2|^2+|\phi_3|^2-|\phi_4|^2=ds^2.
\end{equation}
Therefore, $\mb{x}^*:M\ra \R^4$ is a minimal surface whenever $\mb{x}: M\ra \R^{3,1}$ is a
space-like stationary surface, and $(ds^*)^2$ is complete whenever $ds^2$ is.
\item Letting
\begin{equation}
\psi_1^*:=\f{\phi_1^*+i\phi_2^*}{\phi_3^*-i\phi_4^*},\quad \psi_2^*:=\f{\phi_1^*-i\phi_2^*}{\phi_3^*-i\phi_4^*},\quad dh^*:=\f{1}{2}(\phi_3^*-i\phi_4^*),
\end{equation}
we see $\mb{x}^*:M\ra \R^4$ and $\mb{x}:M\ra \R^{3,1}$ share the same W-data.
\end{itemize}

Thereby, in conjunction with Fujimoto's theorem \cite{F} on complete minimal surfaces in $\R^4$, we obtain the following Berntein-type theorem:

	\begin{thm}
		Let $\mb{x}:M\ra \R^{3,1}$ be a complete space-like stationary surface, $(\psi_1,\psi_2)$ be the Gauss map of $M$,
and $q_i$ ($i=1$ or $2$) be the number of points in $\C^*$ that $\psi_i$ does not take ($q_i$ could be $0$ or $\infty$).
 If neither $\psi_1$ nor $\psi_2$ are constant, then $\min\{q_1,q_2\}\le 3$ or $q_1=q_2=4$.		
	\end{thm}
	
\begin{proof}
Let $\td{M}$ be the universal covering space of $M$, then $\td{\psi}_i=\psi\circ \pi$ ($i\in \{1,2\}$) implies
the number of exception values of $\td{\psi}_i$ equals $q_i$. Thus, we can assume $M$ is simply connected,
without loss of generality. Let $\mb{x}^*:M\ra \R^4$ be the dual immersion of $\mb{x}$, then it gives a complete minimal surface
in $\R^4$ sharing the same W-data, and hence the conclusion on $q_1,q_2$ immediately follows from Theorem \ref{be}.
\end{proof}

\bigskip\bigskip
	
	\Section{On the Gauss map of degenerate stationary surfaces}{On the Gauss map of degenerate stationary surfaces}
	\label{S4}
	\begin{defi}
		Let $M$ be a space-like stationary surface in $\R^{3,1}$. If the Gauss image $G(M)$ of $M$ lies in a hyperplane $H_A\subset Q_{1,1}$
		with $[A]\in \Bbb{CP}^{2,1}$,
		then $M$ is called {\bf degenerate}. Moreover, $M$ is said to be a {\bf degenerate stationary surface of totally real (or hyperbolic, elliptic,
			parabolic) type} whenever $[A]$ is totally real (or hyperbolic, elliptic, parabolic). If the Gauss image of $M$ lies in the intersection of $k$ linear independent
		hyperplanes in $Q_{1,1}$, then we say $M$ is {\bf $\mb{k}$-degenerate}.
	\end{defi}
	
Let $M$ be a degenerate stationary surface, then $G(M)\subset H_A$ and
$G(M)\subset Q_{1,1}^+$ implies
\begin{equation}
G(M)\subset H_A^+.
\end{equation}
	Assume $[A]\sim [B]$ and let $\si\in SO^+(3,1)$ sending $[A]$ to $[B]$, then $\si(M)$ is congruent to $M$
	and its Gauss image lies in $\si(H_A^+)=H_{B}^+$. Therefore, it suffices for us to consider the representative hyperplanes
under the action of $SO^+(3,1)$.
Guided by Theorem \ref{class3}, we shall investigate all types of degenerate stationary surfaces case by case in the following text.
	\subsection{Degenerate stationary surfaces of totally-real type}\label{S4.1}
For the totally-real case, $[A]=[\mb{u}_0]$ implies $\lan \mb{x}-\mb{x}_0,\mb{u}_0\ran\equiv 0$ with $\mb{x}_0$ being the position vector of a fixed point in $M$,
and it follows that:
	\begin{pro}\label{t_r}
		For each degenerate stationary surface $M$ of totally real type whose Gauss image lies in $H_A$, we have:
		\begin{enumerate}
			\item [(1)] $[A]\sim [\mb{u}_S]$ if and only if $M$ is congruent to a maximal surface in $\R^{2,1}$.
			\item [(2)] $[A]\sim [\mb{u}_T]$ if and only if $M$ is congruent to a minimal surface in $\R^3$.
			\item [(3)] $[A]\sim [\mb{u}_N]$ if and only if $M$ is congruent to a zero mean curvature surface in $\R^{2,0}:=\{\mb{x}\in \R^{3,1}:x_3=x_4\}$
			endowed with the induced degenerate inner product.
		\end{enumerate}
	\end{pro}
	
	The following theorem give characteristics for zero mean curvature surfaces in $\R^{2,0}$.
	\begin{thm}\label{t_r2}
		For each space-like stationary surface $M$ in $\R^{3,1}$, the following statements are equivalent:
		\begin{enumerate}
			\item [(a)] The Gauss image of $M$ lies in $H_A$ with $[A]\sim [\mb{u}_N]$.
			\item [(b)] The Gauss image of $M$ lies in $H_B$ with $[B]\in Q_{1,1}^+$.
			\item [(c)] Either $\psi_1$ or $\psi_2$ is a constant function on $M$.
			\item [(d)] The Gauss curvature $K$ of $M$ is $0$ everywhere.
			\item [(e)] $M$ is $2$-degenerate.
		\end{enumerate}
		
		Moreover, a complete surface satisfying the above conditions has to be the entire graph of $F:(x_1,x_2)\in \R^2\mapsto h(x_1,x_2)\mb{y}_0\in \R^{1,1}$,
		where $h$ is a harmonic function and $\mb{y}_0$ is a null vector. In this case, $\psi_i$ ($i=1$ or $2$) omits $1$ or $2$ points in $\C^*$ whenever $\psi_i$ is not constant.
		
	\end{thm}
	
	\begin{proof}
		$(a)\Rightarrow (b)$ can be got by an easy deduction based on (\ref{phi1}). $(b)\Rightarrow (c)$ immediately follows from Proposition \ref{hp2} and
		the holomorphicity of the Gauss map. $(c)\Rightarrow (d)$ is a direct corollary of (\ref{K}). $(e)\Rightarrow (a)$ can
		be directly deduced from Corollary \ref{hp3}. It remains to show $(d)\Rightarrow (e)$.
		
		Assume $M$ is a space-like stationary surface satisfying $K\equiv 0$. If the height differential
		$dh\equiv 0$, then (\ref{phi6}) means
		the Gauss image of $M$ lies in $H_A\cap H_B$ with $[A]=[(0,0,1,-1)]$, $[B]=[(1,-i,0,0)]$ or $[(1,i,0,0)]$
		and hence $M$ is $2$-degenerate. Otherwise, the curvature formula is
		given as (\ref{K}) in term of the W-data $(\psi_1,\psi_2,dh)$. If $\psi_1\equiv c_1$, Corollary \ref{hp3} enables us to find $[A],[B]\in Q_{1,1}$, so that
		the Gauss image of $M$ lies in $H_A\cap H_B$. Otherwise, the zeros of $\psi'_1$ are discrete. Then we can find
		a neighborhood $U$ of $p\in M$, such that the restriction of $\psi_1$ on $U$ gives a biholomorphism between $U$
		and a neighborhood $\Om$ of $\psi_1(p)\in \C^*$. Observing that $\psi_2\neq \bar{\psi}_1$, we can choose a M\"{o}bius transformation
		$\mc{M}_T$, such that $\mc{M}_T(\psi_1(p))=0$, $\mc{M}_{\bar{T}}(\psi_2(p))=i$; Let $\si$ be the corresponding Lorentz transform (see \S \ref{conj}), then the Gauss image
		of $\si(p)\in \si(M)$ is $\Psi^{-1}(0,i)$; Therefore, we can assume $\psi_1(p)=0$, $\psi_2(p)=i$ without loss of generality, and
		$\psi_2$ can be written as $f\circ \psi_1$ on $U$, with a holomorphic function $f$ on $\Om$ satisfying $f(0)=i$.
		Applying the Leibniz chain rule, we have $\psi'_2=f'(\psi_1)\psi'_1$, and then substituting it into (\ref{K}) implies
		\begin{equation}\label{f1}
		\text{Re}[f'(z)(z-\bar{f}(z))^2]=0\qquad \text{on }\Om.
		\end{equation}
		
		Denote
		\begin{equation}
		f(z)=i+a_1z+a_2z^2+\cdots+a_kz^k+\cdots,
		\end{equation}
		then
		\begin{equation*}\aligned
		&f'(z)(z-\bar{f}(z))^2\\
		=&(a_1+2a_2z+\cdots)(i+z-\bar{a}_1\bar{z}-\cdots)^2\\
		=&-a_1+((2ia_1-2a_2)z-2i|a_1|^2\bar{z})+(-2|a_1|^2-4ia_2\bar{a}_1)z\bar{z}+\cdots
		\endaligned
		\end{equation*}
		and (\ref{f1}) forces
		$$\left\{\begin{array}{c}
		a_1\in i\R\\
		2ia_1-2a_2=-2i|a_1|^2\\
		-2|a_1|^2-4ia_2\bar{a}_1\in i\R
		\end{array}\right.$$
		Through a straightforward calculation we get $a_1=0$. Next, under the inductive assumption $a_1=\cdots=a_{k-1}=0$
		for any $k\geq 2$, we have
		\begin{equation*}\aligned
		&f'(z)(z-\bar{f}(z))^2\\
		=&(ka_kz^{k-1}+\cdots)(i+z-\bar{a}_k\bar{z}^k-\cdots)^2\\
		=&(-ka_kz^{k-1}+0\cdot \bar{z}^{k-1})+\cdots
		\endaligned
		\end{equation*}
		and hence $a_k=0$. Therefore $f\equiv i$ and $\psi_2$ has to be a constant function. Then we can proceed as above to
		conclude that $M$ has to be $2$-degenerate. This completes the proof of $(d)\Rightarrow (e)$.
		
		Let $\mb{x}:M\ra \R^{3,1}$ be a complete space-like stationary surface, satisfying $x_3\equiv x_4$. Since $K\equiv 0$, the universal covering space $\td{M}$
		of $M$ is conformally equivalent to $\C$.
		 $\td{\phi}_1^2+\td{\phi}_2^2+\td{\phi}_3^2-\td{\phi}_4^2=0$
		along with $\td{x}_3=\td{x}_4$ forces $\td{\phi_3}=\td{\phi}_4$ and
		$\td{\phi}_2=i\td{\phi}_1$ or $-i\td{\phi}_1$. Denote $\td{\phi}_i=f_i dz$ with an entire function $f_i$ for each
$1\leq i\leq 4$, then
		$$d\td{s}^2=|\td{\phi}_1|^2+|\td{\phi}_2|^2+|\td{\phi}_3|^2-|\td{\phi}_4|^2=2|f_1|^2 |dz|^2$$
		and hence $f_1$ has no zero. The completeness of $M$ implies
		$$L(\g)=\int_\g d\td{s}=\sqrt{2}\int_\g |f_1|dz=+\infty$$
		for each divergent path $\g:[0,+\infty)\ra \td{M}$,
		and then Lemma 9.6 of \cite{O-Survey} tells us $f_1$ is either constant or has a pole at $\infty$, where the latter forces the existence of a zero of $f_1$ and causes a contradiction. Therefore $f_1\equiv c$ with $c\neq 0$. Let $w:=2cz$, then
\begin{equation}\label{phi7}
\td{\phi}=\f{1}{2}[(1,\pm i,f,f)]dw
\end{equation}
 with $f$ an entire function.
		Integrating both sides of (\ref{phi7}) shows the image of $\mb{\td{x}}$ forms the entire graph of $F:(x_1,x_2)\mapsto h(x_1,x_2)\mb{y}_0$
		with $h$ a harmonic function, and $M=\td{M}$ follows from the injectivity of $\mb{\td{x}}$. Finally, (\ref{phi7}) also implies
the non-constant one of $\psi_1,\psi_2$ equals $f^{-1}$, which take each values of $\C^*$ with the exception of $1$ or $2$ points, due to Picard's theorem.
		
	\end{proof}

\noindent{\bf Remark.} By Theorem \ref{t_r2}, a space-like stationary surface $\mb{x}: M\ra \R^{3,1}$ is flat if and only if the dual immersion $\mb{x}^*: M\ra \R^4$
gives a 2-degenerate minimal surfaces in $\R^4$, i.e.  a complex curve in $\C^2$ (see Proposition 4.6 of \cite{H-O}). Moreover, the completeness
of such surface implies $\psi_i^*$ omits at most 3 points whenever it is nonconstant (see
Theorem \ref{be}).
This result is optimal since we can construct the following
examples of complete minimal surfaces, such that the numbers of exceptional values of $\psi_1^*$ are 0,1,2,3, respectively:
\begin{itemize}
\item $\mb{x}_0^*:\C\ra \R^4$ given by the W-data $\psi_1^*=\f{z^2-1}{z}$, $\psi_2^*\equiv 0$ and $dh=zdz$;
\item $\mb{x}_1^*:\C\ra \R^4$ given by $\psi_1^*=\f{1}{z}$, $\psi_2^*\equiv 0$ and $dh=zdz$;
\item $\mb{x}_2^*:\C\ra \R^4$ given by $\psi_1^*=e^{-z}$, $\psi_2^*\equiv 0$ and $dh=e^z dz$;
\item $\mb{x}_3^*:\Bbb{D}\ra \R^4$ given by $\psi_1^*=f(z)$, $\psi_2^*\equiv 0$ and
$dh=\f{f'(z)}{f(z)(f(z)-1)}dz$, with $f$ a holomorphic covering map from the unit disc $\Bbb{D}$
onto $\C\backslash \{0,1\}$.
\end{itemize}
For $i=1$ or $2$, the dual immersion $\mb{x}_i:\C\ra \R^{3,1}$ of $\mb{x}_i^*$ gives a complete space-like
stationary surface. On the other hand, $\mb{x}_0:\C\ra \R^{3,1}$ cannot be a space-like surface since $\psi_2\neq \bar{\psi}_1$ does not hold everywhere;
 the metric induced by $\mb{x}_3:\Bbb{D}\ra \R^{3,1}$ cannot be complete, due to the flatness.
	
	\subsection{Degenerate stationary surfaces of hyperbolic type}\label{graph2}
	
	
For the hyperbolic case, as shown in Theorem \ref{class1}, $[A]\sim [(\tanh u,i,0,0)]$ with $u\in (0,+\infty]$. Here $[A]\in Q_{1,1}^+$ if and only if $u=+\infty$ and the corresponding
surface $M$ is 2-degenerate, which has been considered in \S \ref{S4.1}. Thereby, we assume $u\in (0,+\infty)$ in this subsection and call $u$ the {\it hyperbolic argument}
	of $M$.
	
	By Theorem \ref{class3}, there exists $S\stackrel{conj}{\sim}\begin{pmatrix}
	e^u & 0\\
	0 & e^{-u}
	\end{pmatrix}$, such that $\Psi(H_A)$ ($\Psi(H_A^+)$) is the graph of $\mc{M}_S$
over $\C^*$ ($\C^*\backslash E_S$),
	where $E_S$ consists of exact 2 points corresponding
	to the 2 conjugate eigenvectors of $S$ that are $\C$-linear independent (see Proposition \ref{conj2}).
Without loss of generality we can assume
	\begin{equation}
	S=\begin{pmatrix}
	e^u & 0\\
	0 & e^{-u}
	\end{pmatrix}
	\end{equation}
	then
\begin{equation}
\mc{M}_S(w)=e^{2u}w,\quad E_S=\{0,\infty\}.
\end{equation}
	Let
	\begin{equation}
	\psi:=\psi_1,
	\end{equation}
	then
$\psi$ is
	a holomorphic function with no zero, $\psi_2=e^{2u}\psi$, and (\ref{phi4}) and (\ref{phi5}) become
	\begin{equation}
	\phi=((1+e^{2u})\psi,-i(1-e^{2u})\psi,1-e^{2u}\psi^2,1+e^{2u}\psi^2)dh,
	\end{equation}
	\begin{equation}\label{metric1}
	ds^2=2|\bar{\psi}-e^{2u}\psi|^2|dh|^2.
	\end{equation}
	Thereby, we get the Weierstrass representation for degenerate stationary surfaces of hyperbolic type as follows:
	
	\begin{thm}\label{W1}
		Given a holomorphic 1-form $dh$ and a holomorphic function $\psi$ globally defined on a Riemann surface $M$, if
		\begin{itemize}
			\item $dh$ and $\psi$ has no zero,
			\item $\oint_\g \psi dh=0$ and $\text{Re}\oint_\g dh=\text{Re}\oint_\g \psi^2 dh=0$ for each closed path
			in $M$,
		\end{itemize}
		then
		\begin{equation}
		\mb{x}:=\text{Re}\int ((1+e^{2u})\psi,-i(1-e^{2u})\psi,1-e^{2u}\psi^2,1+e^{2u}\psi^2)dh
		\end{equation}
		defines a degenerate stationary surface with the hyperbolic argument $u\in (0,+\infty)$. Conversely, all nonflat degenerate stationary
		surfaces of hyperbolic type can be expressed in this form.
		
	\end{thm}
	
	Observing that the conditions in Theorem \ref{W1} is independent of $u$, each degenerate stationary surface
	with the W-data $\psi,dh$ and the hyperbolic argument $u_0\in (0,+\infty)$ can be deformed to a family
	of degenerate stationary surfaces $\{M_u:u\in (0,+\infty)\}$ of hyperbolic type, keeping $\psi$ and $dh$
	invariant. Moreover, let
	\begin{equation}\label{metric3}
	d\hat{s}^2:=|\psi|^2|dh|^2
	\end{equation}
	be a Riemannian metric independent of $u$, then $d\hat{s}^2$ is complete if and only if $ds^2$ is complete, since
	\begin{equation}
	(e^{2u}-1)|\psi|\leq |\bar{\psi}-e^{2u}\psi|\leq (e^{2u}+1)|\psi|.
	\end{equation}
	Thus the deformation preserves the completeness property.

	As shown in \cite{M-W-Y}, each entire space-like stationary graph $M$ of $F:\R^2\ra \R^{1,1}$ has to be a complete degenerate stationary surface of hyperbolic type.
	More precisely, given such graph, there exists a non-singular transform
	\begin{equation}
	\aligned
	x_1&=u,\\
	x_2&=au+bv\quad (b>0),
	\endaligned
	\end{equation}
	such that $(u,v)$ are global isothermal parameters of $M$; Letting $z:=u+iv$
	we have $\phi_1=\f{1}{2}dz$, $\phi_2=\f{c}{2}dz$ with $c:=a-ib$, and hence
	the Gauss image of $M$ lies in $H_A$ with $A=[(c,-1,0,0)]$,
	and $[A]\notin Q_{1,1}^+$ if and only if $c\neq \pm i$. Moreover, a straightforward computation as in \cite{M-W-Y} gives
	$$\phi=(\f{1}{2},\f{c}{2},\mu\cosh \be,\mu\sinh\be)dz,$$
	where $\mu^2=-\f{1+c^2}{4}$ and $\be$ is an entire function,
	$$ds^2=\lan \phi,\bar{\phi}\ran\geq \f{1+|c|^2-|1+c^2|}{4}|dz|^2$$
	is a complete metric, and
	\begin{equation}
	\psi_1=\f{\phi_1+i\phi_2}{\phi_3+\phi_4}=\f{1+ic}{2\mu}e^{-\be}
	\end{equation}
	takes each value of $\C\backslash \{0\}$ for infinitely times, unless $M$ is an affine
	space-like plane. Therefore, $\int_M K dA_M=\int_M G^* dA_g$ (where $A_g$ is the metric form on $H_A^+$) should be divergent provided that $M$
	is non-flat. This is an alternative proof of Theorem 6.1 in \cite{M-W-Y}.
	
	Besides the above examples, can we construct another complete degenerate stationary surfaces of hyperbolic type?
	The answer is 'No':
	
	\begin{thm}
		\label{grph}
$M\subset \R^{3,1}$ is a complete degenerate stationary surface of hyperbolic type
if and only if it is congruent to the entire space-like stationary graph of $F:\R^2\ra \R^{1,1}$.
		Moreover, if $M$ is non-flat, then
		$\psi_i$ ($i=1$ or $2$) takes each point in $\C^*$ for infinitely times, with the exception of exactly 2 points, and the total Gauss curvature of $M$
		is $\infty$.
	\end{thm}
\noindent {\bf Remark. }Here, 'complete degenerate stationary surfaces of hyperbolic type' include 2-degenerate ones (i.e. $[A]\in Q_{1,1}^+$), which are just the entire graphs of harmonic functions over $\R^2$, as shown by Theorem \ref{t_r2}. So this case will not be mentioned in the following proof.

	\begin{proof}
Let $\pi:\td{M}\ra M$ be the universal covering map,
then the completeness of $d\hat{s}^2$ given in (\ref{metric3}) implies
\begin{equation}
\pi^* d\hat{s}^2=|\td{\psi}|^2 |d\td{h}^2
\end{equation}
is a complete metric on $\td{M}$, where $\td{\psi}:=\psi\circ \pi$, $d\td{h}:=\pi^* dh$ constitute the W-data of $\td{M}$.

		If $\td{M}$ is conformally equivalent to the unit disk $\Bbb{D}$, we can write $d\td{h}=fdz$, where $f$ is a holomorphic function
		on $\Bbb{D}$ with no zero. By Lemma 8.5 of \cite{O-Survey}, there is a divergent path $\g$ in $\Bbb{D}$, such that
		\begin{equation}
		\int_\g \pi^*d\hat{s}=\int_\g |\td{\psi} f||dz|<+\infty,
		\end{equation}
		which contradicts to the completeness of $\pi^*d\hat{s}^2$. Thus $\td{M}$ should be conformally equivalent to $\C$. Again using Lemma 9.6
		of \cite{O-Survey}, we can proceed as in the proof of Theorem \ref{t_r2} to conclude that $\td{\psi} f\equiv \text{const}$ and hence
		$\td{\phi}_1=c_1dz$, $\td{\phi}_2=c_2dz$ with $c_1,c_2$ being both complex constants. Therefore, $\td{\mb{x}}:\td{M}\ra \R^{3,1}$ gives an entire stationary graph of a function from $\R^2$ into $\R^{1,1}$, and the injectivity of $\td{\mb{x}}$ implies $M=\td{M}$.
		
	\end{proof}

	\subsection{Degenerate stationary surfaces of elliptic type}\label{elliptic}
	
	
 For the elliptic case,
  Theorem \ref{class1} tells us
  $[A]\sim [(0,0,1,i\tan\a)]$, where $\a\in (0,\f{\pi}{2})$ is called the {\it elliptic argument}
	of $M$.
	
	As shown in Theorem \ref{class3}, there exists $S\stackrel{conj}{\sim} \begin{pmatrix}
	0 & ie^{-i\a}\\
	ie^{i\a} & 0
	\end{pmatrix}$, such that $\Psi(H_A)$ is a graph of $\mc{M}_S$ over $\C^*$, and $H_A^+=H_A$. Without loss of generality we can assume
	\begin{equation}
	S=\begin{pmatrix}
	0 & ie^{-i\a}\\
	ie^{i\a} & 0
	\end{pmatrix}
	\end{equation}
	then
\begin{equation}
\mc{M}_S(w)=e^{-2i\a}w^{-1}.
\end{equation}
	Let
	\begin{equation}
	\psi:=\psi_1,
	\end{equation}
	then $\psi$ is
	a meromorphic function, $\psi_2=e^{-2i\a}\psi^{-1}$, and (\ref{phi4}) and (\ref{phi5}) become
	\begin{equation}\aligned
	\phi&=(\psi+e^{-2i\a}\psi^{-1},-i(\psi-e^{-2i\a}\psi^{-1}),1-e^{-2i\a},1+e^{-2i\a})dh\\
	&=(e^{i\a}\psi+e^{-i\a}\psi^{-1},-i(e^{i\a}\psi-e^{-i\a}\psi^{-1}),2i\sin\a,2\cos\a)\om
	\endaligned
	\end{equation}
	with $\om:=e^{-i\a}dh$ and
	\begin{equation}\label{metric2}
	ds^2=2|\bar{\psi}-e^{-2i\a}\psi^{-1}|^2|dh|^2=2\Big[|\psi|^2+|\psi|^{-2}-2\cos(2\a)\Big]|\om|^2,
	\end{equation}
	respectively. This enables us to obtain the Weierstrass representation for degenerate stationary surfaces of elliptic type:
	
	\begin{thm}\label{W3}
		Given a holomorphic 1-form $\om$ and a meromorphic function $\psi$ globally defined on a Riemann surface $M$, if
		\begin{itemize}
			\item $p$ is a zero of $\om$ if and only if $p$ is a zero or a pole of $\psi$, with the same order;
			\item $\oint_\g \psi \om=-\overline{\oint_\g \psi^{-1}\om}$ and $\oint_\g \om=0$ for each closed path
			in $M$,
		\end{itemize}
		then
		\begin{equation}\label{W2}
		\mb{x}:=\text{Re}\int (e^{i\a}\psi+e^{-i\a}\psi^{-1},-i(e^{i\a}\psi-e^{-i\a}\psi^{-1}),2i\sin\a,2\cos\a)\om
		\end{equation}
		defines a degenerate stationary surface with the elliptic argument $\a\in (0,\f{\pi}{2})$. Conversely, all degenerate stationary
		surfaces of elliptic type can be expressed in this form.
		
	\end{thm}
	
	Similarly to the hyperbolic case, the conditions on the W-data is independent of $\a$. Hence each $M$
	with the elliptic argument $\a_0$  naturally induces a family of degenerate stationary surfaces $\{M_\a:\a\in (0,\f{\pi}{2})\}$
	via the invariant W-data $(\psi,\om)$. It is worthy to note that, letting $\a:=\f{\pi}{2}$ in (\ref{W2}) defines a minimal
	surface $M_{\f{\pi}{2}}\subset \R^3$, which is called the {\it associated minimal surface in $\R^3$}. Conversely,
	given a minimal surface $M_{\f{\pi}{2}}\subset \R^3$ given by
	\begin{equation}\label{W4}
	\mb{x}:=\text{Re}\int i(\psi-\psi^{-1}),\psi+\psi^{-1},2i,0)\om,
	\end{equation}
	where $(\psi,\om)$ is the W-data satisfying the conditions of Theorem \ref{W3}, we can define the deformation (\ref{W2}) with
	the parameter $\a\in (0,\pi/2)$, yielding a family of degenerate stationary surfaces of elliptic type. This is so called
{\it Al\'{\i}as-Palmer deformation} \cite{A-P}. Let
	\begin{equation}
	d\hat{s}^2:=(|\psi|^2+|\psi|^{-2}|)|w|^2,
	\end{equation}
	be a metric independent of the parameter $\a$,
	then $|\psi|^2+|\psi|^{-2}\geq 2$ implies
	\begin{equation}
	2(1-|\cos(2\a|))d\hat{s}^2\leq ds^2\leq 2(1+|\cos(2\a)|)d\hat{s}^2,
	\end{equation}
	hence this deformation preserves the completeness property. This observation along with the results
	of K. Voss \cite{V} and H. Fujimoto \cite{F} on the exceptional values of Gauss maps for nonflat complete minimal surfaces in
$\R^3$ enable us to get the following conclusion:
	
	\begin{thm}\label{be2}
		Let $E$ be a subset of $\C^*$, $\a\in (0,\f{\pi}{2})$, then
		\begin{itemize}
			\item If $E$ is a finite set of no more than $4$ points, then there is a complete degenerate stationary surface $M$
			with the elliptic argument $\a$, such that the image of $\psi$ omits exactly those points in $E$.
			\item If $E$ contains at least $5$ points, then any complete degenerate stationary surface $M$ of elliptic type
			has to be affine linear, whenever the image of $\psi$ is contained in $\C^*\backslash E$.
		\end{itemize}
	\end{thm}
	
	Now we assume $M$ is complete and has finite total Gauss Curvature. Due to Huber's theorem \cite{Hu}, $M$ is conformally equivalent to $\bar{M}\backslash \{p_1,\cdots,p_r\}$
	with $\bar{M}$ a compact Riemann surface. If $p_i$ is an essential singularity of $\psi$, then according to Picard's big theorem,
on any neighborhood of $p_i$,
the Gauss map $G$ takes each values of $H_A^+\cong \C^*$ infinitely often, except for at most 2 points, which forces the infinity of the area of the Gauss image,
	and also the divergence of the total Gauss curvature. Therefore, $\psi$ can be extended to a meromorphic function on $\bar{M}$,
	which takes each values of $\C^*$ for a fixed number of times. In conjunction with (\ref{area}) we get the following result:
	\begin{pro}\label{totalcur}
		Let $M$ be a complete degenerate stationary surface of elliptic type, then the total Gauss curvature of $M$ is either $\infty$
		or $-4\pi m$, where $m$ equals the degree of $\psi$.
		
	\end{pro}
	
\noindent {\bf Remark. }On the premise that $M$ is algebraic, i.e. $\phi=\mb{x}_zdz$ can be extended to a vector-valued meromorphic form on $\bar{M}$, we conclude that all ends of $M$ has to be regular, i.e. $\psi_2\neq \bar{\psi}_1$ holds everywhere on
$\bar{M}$ (see \cite{M-W-W}),
and then $\int_M KdA_M=-4\pi \text{deg }\psi$
follows directly from the Gauss-Bonnet type formula in \cite{M-W-W}.  However, this premise does not always hold
for an arbitrary complete space-like stationary surface in $\R^{3,1}$ with finite total Gauss curvature, due to the counterexamples
constructed by Ma-Wang-Wang \cite{M-W-W}.


	Let $F:(u_1,u_2)\in \R^{1,1}\mapsto (F_1(u_1,u_2),F_2(u_1,u_2))\in \R^2$ be a vector-valued function, such that
	\begin{equation}
	M:=\{(F_1(u_1,u_2),F_2(u_1,u_2),u_1,u_2):(u_1,u_2)\in \R^{1,1}\}
	\end{equation}
	is a space-like stationary graph. Then the induced metric $ds^2=g_{ij}du_idu_j$ on $M$ with $g_{ij}:=\lan \mb{x}_{u_1},\mb{x}_{u_2}\ran$
	has to be positive definite, and $u_1,u_2$ are both harmonic functions. In conjunction with the expression of the Laplace operations on $M$,
	we can proceed as in \cite{M-W-Y} to find smooth functions $\xi_1,\xi_2$ on $\R^{1,1}$, such that
	\begin{equation}
	\left(\pd{\xi_i}{u_j}\right)=W^{-1}(g_{ij})\qquad \text{with }W:=\det(g_{ij})^{\f{1}{2}}.
	\end{equation}
	Define the Lewy's transformation (see \cite{Le}) $L:(x_1,x_2)\in \R^{1,1}\ra (\eta_1,\eta_2)\in \R^{1,1}$ by
	\begin{equation}
	\eta_i=u_i+\xi_i(u_1,u_2),\quad i=1,2,
	\end{equation}
	then applying the same technique in \cite{M-W-Y}, we conclude that $L$ is a diffeomorphism of $\R^{1,1}$ onto itself,
	and $(\eta_1,\eta_2)$ are global isothermal parameters on $M$. Denote $\ze:=\eta_1+i\eta_2$, then
	\begin{equation}
	f_i:=\pd{u_i}{\ze}=\f{1}{2}\Big(\pd{u_i}{\eta_1}-i\pd{u_i}{\eta_2}\Big)\qquad i=1,2
	\end{equation}
	are both entire functions, satisfying $\text{Im}(f_1\bar{f}_2)=\pd{(u_1,u_2)}{(\eta_1,\eta_2)}>0$
	everywhere. Noting that $f_1\bar{f}_2=\f{f_1}{f_2}\cdot |f_2|^2$, we conclude $f_2=cf_1$ with $\text{Im}(c)<0$
	and $f_1$ has no zero, due to the Liouville's theorem. Let $c:=a-ib$, then taking the real and imaginary parts
	of $f_2=cf_1$ gives
	\begin{equation*}
	\aligned
	\pd{u_2}{\eta_1}&=a\pd{u_1}{\eta_1}-b\pd{u_1}{\eta_2},\\
	\pd{u_2}{\eta_2}&=b\pd{u_1}{\eta_1}+a\pd{u_1}{\eta_2}.
	\endaligned
	\end{equation*}
	Now we let $(u,v)$ be global parameters of $M$,
	such that there is a non-singular transformation
	\begin{equation}
	\aligned
	u_1&=u,\\
	u_2&=au+bv,
	\endaligned
	\end{equation}
	then
	$$\pd{v}{\eta_1}=-\pd{u}{\eta_2},\quad \pd{v}{\eta_2}=\pd{u}{\eta_1}$$
	and hence $(u,v)$ are global isothermal parameters.
	Using the complex parameter $z:=u+iv$, we have
\begin{equation}
\phi=(h_1,h_2,\f{1}{2},\f{c}{2})dz
\end{equation}
 with $h_1$ and $h_2$ being both entire functions.
	$\lan \phi,\phi\ran=0$ implies
	\begin{equation}
	h_1=\mu \cos\be,\qquad h_2=\mu \sin\be,
	\end{equation}
	where $\mu^2=\f{c^2-1}{4}$ and $\be$ is an entire function, and
	\begin{equation}
	ds^2=\lan \phi,\bar{\phi}\ran=\Big(|\mu|^2\cosh(2\text{Im}\be)+\f{1}{4}-\f{|c|^2}{4}\Big)|dz|^2\geq \f{|c^2-1|+1-|c|^2}{4}|dz|^2
	\end{equation}
	is a complete metric. The expression of $\phi$ means the Gauss image of $M$ lies
	in $H_A$ with $[A]=[(0,0,c,1)]$, and
	$$\psi_1=\f{\phi_1+i\phi_2}{\phi_3+\phi_4}=\f{2\mu}{1+c}e^{i\be}$$
	takes each values of $\C\backslash \{0\}$ for infinitely times unless $\be$ is a constant function. This implies
	$\int_M KdA_M$ is divergent. We rewrite the above conclusions as the following theorem:
	
	\begin{thm}\label{graph1}
		Each entire space-like stationary graph $M$ of $F:\R^{1,1}\ra \R^2$ is a complete degenerate stationary surface of elliptic type.
		Moreover, if $M$ is nonflat, then
		$\psi_i$ ($i=1$ or $2$) takes each point in $\C^*$ for infinitely times, with the exception of exactly 2 points, and the total Gauss curvature of $M$
		is $\infty$.
	\end{thm}
	
	Unlike the hyperbolic case, a complete space-like stationary graph $M$ of $F: D\subset \R^{1,1}\ra \R^2$ can be neither an entire graph, nor a surface with infinite total
	Gauss curvature, even if we assume $M$ is a degenerate stationary surface of elliptic type. A counterexample has been given in \cite{M-W-W}. This is an complete graph over a punctured time-like plane given by
	\begin{equation}\label{graph-time2}
	\phi=(z^n+\f{1}{z^n},-i(z^n-\f{1}{z^n}),\sqrt{3}i,1)\qquad z\in \C\backslash \{0\},n\in \Bbb{Z},n\geq 2,
	\end{equation}
	and the total Gauss curvature is $-4\pi n$. Such examples can be characterized as follows:
	
	\begin{thm}\label{graph-time}
		Let $M$ be a complete space-like stationary surface, satisfying:
		\begin{itemize}
			\item $M$ is a graph over a domain of a time-like plane;
			\item $M$ has finite total Gauss curvature.
		\end{itemize}
		Then $M$ has to be a degenerate stationary surface with the
		elliptic argument $\a\in (0,\f{\pi}{2})$,  given by (\ref{W2}) with respect to the W-data
		\begin{equation}
		\psi=h,\quad \om=dz
		\end{equation}
		over $\{z\in \C: h(z)\in \C\backslash \{0\}\}$. Here $h$ is a rational function on $\C$, such that
	 the residue on each pole of $h$ or $h^{-1}$ equals $0$. And $\int_M KdA_M=-4\pi \text{deg}(h)$.
		
	\end{thm}
	
	\begin{proof}
		$M$ is a graph means the projection onto the time-like plane is injective, hence the genus of $M$
		is $0$. Applying Huber's theorem, $M$ is conformally equivalent to $\{\ze\in \C: \ze\neq \ze_i, \forall i=1,\cdots,m\}$.
		Let $\eta_1:=\text{Re}\ze$, $\eta_2:=\text{Im}\ze$, $\phi_3=f_1d\ze$, $\phi_4=f_2d\ze$, then
		$\Im(f_1\bar{f}_2)=\pd{(x_3,x_4)}{(\eta_1,\eta_2)}\neq 0$ everywhere, i.e. $\f{f_2}{f_1}$ takes values in a half plane.
		The Picard theorem says $f_2=cf_1$ with $\text{Im }c\neq 0$ and $f_1$ has no zero. Afterwards, we can proceed as in the proof
		of Theorem \ref{graph1} to find a global isothermal parameters $(u,v)$ of $M$, such that there exists a non-singular transformation between
		$(x_3,x_4)$ and $(u,v)$. Letting $z:=\la(u+iv)$ with a suitable complex constant $\la$ and making a suitable $SO^+(3,1)$-action, we can assume the
		Weierstrass representation of $M$ is given in (\ref{W2}) with $\psi=h$, $\om=dz$, and the parameter domain is replaced by
		$\C\backslash \{z_1,\cdots,z_m\}$. Finally, the completeness of $M$ equals to say $z_1,\cdots,z_m$ are exactly the zeros and poles of $h$,
		and the period condition of $M$ holds if and only if the residue on each pole of $h$ or $h^{-1}$ is $0$.

	\end{proof}
	
\noindent {\bf Remark. }Given a rational function $f$, the residue on every pole of $f$ all equals $0$ if and only $f$ is the derivative of another rational function.
Hence $h(z):=z^n$ with $n\geq 2$ satisfies the conditions of Theorem \ref{graph-time}, which induces the known example (\ref{graph-time2}).
In fact, if $h$ is a polynomial, denote
$$\f{1}{h}=\left(\f{P}{Q}\right)'=\f{P'Q-PQ'}{Q^2}$$
with polynomials $P,Q$ that are relatively prime to each other, then $z\in \C$ is a zero of $P'Q-PQ'$ of order $r$ if and only if
it is the zero of $Q$ of order $r+1$; Comparing the degrees of the above 2 polynomials implies $P\equiv \text{const}$, $Q$ has exactly 1 zero and hence $h(z)=c(z-z_0)^n$ with $c\neq 0$
and $n\geq 2$. However, without this assumption, we can construct a constellation of such functions, for example, putting
\begin{equation}
h(z):=\f{(z^{n+1}-a)^m}{z^n}\qquad \text{with }a\neq 0,n,m\geq 2
\end{equation}
ensures $h$ and $h^{-1}$ are both derivatives of rational functions.
	

	\subsection{Degenerate stationary surfaces of parabolic type}\label{parabolic}
	
	
	For the parabolic case, as shown in Theorem \ref{class3} and Proposition \ref{conj2}, there exists $S\stackrel{conj}{\sim} \begin{pmatrix}
	1 & 1\\
	0 & 1
	\end{pmatrix}$, such that $\Psi(H_A)$ ($\Psi(H_A^+)$) is the graph of $\mc{M}_S$ over $\C^*$ ($\C^*\backslash E_S$), where
	$E_S$ consists of exact 1 point corresponding
	to the unique conjugate eigenvector of $S$ (up to scaling). Without loss of generality we can assume
	\begin{equation}
	S=\begin{pmatrix}
	1 & 1\\
	0 & 1
	\end{pmatrix}
	\end{equation}
	then
\begin{equation}
\mc{M}_S(w)=w+1,\quad E_S=\{\infty\}.
\end{equation}
	Let
	\begin{equation}
	\psi:=\psi_1,
	\end{equation}
	then $\psi$ is
	a holomorphic function, $\psi_2=\psi+1$ and (\ref{phi4}) and (\ref{phi5}) become
	\begin{equation}
	\phi=(2\psi+1,i,1-\psi-\psi^2,1+\psi+\psi^2)dh,
	\end{equation}
	\begin{equation}\label{metric5}
	ds^2=2(1+4(\text{Im}\psi)^2)|dh|^2.
	\end{equation}
	Thereby, we get the Weierstrass representation for degenerate stationary surfaces of parabolic type as follows:
	
	\begin{thm}
		Given a holomorphic 1-form $dh$ and a holomorphic function $\psi$ globally defined on a Riemann surface $M$, if
		\begin{itemize}
			\item $dh$ has no zero,
			\item $\oint_\g dh=0$ and $\text{Re}\oint_\g \psi dh=\text{Re}\oint_\g \psi^2 dh=0$ for each closed path
			in $M$,
		\end{itemize}
		then
		\begin{equation}\label{W5}
		\mb{x}:=\text{Re}\int (2\psi+1,i,1-\psi-\psi^2,1+\psi+\psi^2)dh
		\end{equation}
		defines a degenerate stationary surface of parabolic type. Conversely, all degenerate stationary
		surfaces of parabolic type can be expressed in this form.
		
	\end{thm}
	
	
	Let
	\begin{equation}
	\R_+^{1,0}:=\R \ep_1\oplus \R (\ep_3+\ep_4),\quad \R_-^{1,0}:=\R \ep_2\oplus \R(\ep_3-\ep_4)
	\end{equation}
	be light-like planes of $\R^{3,1}$. Assume $F:(u_1,u_2)\in \R_+^{1,0}\ra (F_1(u_1,u_2),F_2(u_1,u_2))\in \R_-^{1,0}$
	is a vector-valued function, so that
	\begin{equation}
	M:=\{(u_1,F_1(u_1,u_2),u_2+F_2(u_1,u_2),u_2-F_2(u_1,u_2)):(u_1,u_2)\in \R_+^{1,0}\}
	\end{equation}
	is a space-like stationary graph. We can proceed as in \S \ref{elliptic} to find a global isothermal parameters $(u,v)$, such that
	there is a non-singular transformation
	\begin{equation}\aligned
	u_1&=u,\\
	u_2&=au+bv\quad (b>0).
	\endaligned
	\end{equation}
	Then $z:=u+iv$ is a global complex parameter on $M$, and $\phi_1=\f{1}{2}dz$, $\phi_3+\phi_4=cdz$ with $c:=a-ib$. This implies
	the Gauss image of $M$ lies in $H_A$ with $A=[(-2c,0,1,-1)]$ a parabolic element. Moreover, $\phi_1^2+\phi_2^2
	+\phi_3^2-\phi_4^2=0$ forces
	\begin{equation}
	\phi=\Big(\f{1}{2},h,\f{c}{2}-\f{1}{2c}(\f{1}{4}+h^2),\f{c}{2}+\f{1}{2c}(\f{1}{4}+h^2)\Big)dz
	\end{equation}
	with an entire function $h$. An straightforward calculation shows
	\begin{equation}
	ds^2=\lan \phi,\bar{\phi}\ran=\Big[\f{1}{4}\Big(1-\text{Re}\f{\bar{c}}{c}\Big)+|h|^2-\text{Re}\Big(\f{\bar{c}}{c}h^2\Big)\Big]|dz|^2
	\end{equation}
	is a complete metric on $M$, and
	\begin{equation}
	\psi_1=\f{\phi_1+i\phi_2}{\phi_3+\phi_4}=\f{1}{2c}+\f{ih}{c}.
	\end{equation}
	These formulas along with the divergence of $\int_{H_A^+}dA_g$ (see \S \ref{metric4}) enable us to get the following conclusion:
	
	\begin{thm}\label{graph4}
		Each entire space-like stationary graph $M$ of $F:\R_+^{1,0}\ra \R_-^{1,0}$ is a complete degenerate stationary surface
		of parabolic type. More precisely, $M$ is given by
		\begin{equation*}
		\mb{x}=\text{Re}\int \Big(\f{1}{2},h,\f{c}{2}-\f{1}{2c}(\f{1}{4}+h^2),\f{c}{2}+\f{1}{2c}(\f{1}{4}+h^2)\Big)dz
		\end{equation*}
		over $\C$, where $\text{Im }c<0$ and $h$ is an arbitrary entire function. If $M$ is non-flat, then
		the total Gauss curvature of $M$ is $\infty$. Moreover,
		\begin{itemize}
			\item $M$ is algebraic if and only if $h$ is a polynomial. In this case,
		 $\psi_i$ ($i=1$ or $2$) takes each point in $\C^*$ for exact $n$ times with the exception of exact 1 point, where $n$ is the degree
of $h$.
			\item Otherwise, $\psi_i$ ($i=1$ or $2$) takes each point in $\C^*$ infinitely often, with the exception of 1 or 2 points.
		\end{itemize}

	\end{thm}
	
	It worths to note that, in contrast to the hyperbolic and elliptic cases, an entire stationary graph over a light-like plane
	can be nonflat and algebraic. This type of surfaces can be characterized as follows:
	
	\begin{thm}\label{graph3}
		Each complete, algebraic, degenerate stationary surface in $\R^{3,1}$ of parabolic type has to be an entire graph
		over a light-like plane.
	\end{thm}
	
	\begin{proof}
		Without loss of generality, we can assume $M$ is given by (\ref{W5}) with the W-data $(\psi,dh)$.
		$M$ is algebraic means $M$ is conformally equivalent to $\bar{M}\backslash \{p_1,\cdots,p_m\}$ with $\bar{M}$
		 a compact Riemann surface, and $\psi$ ($dh$) can be extended to a meromorphic function (differential) on $\bar{M}$. The completeness
		of $M$ and the period condition $\oint_\g dh=0$ forces each $p_i$ to be a pole of $dh$ of order $\nu_i\geq 2$. (Note
		that even if $\psi(p_i)=\infty$, we can find a divergent path $\g$ tending to $p_i$, such that $\text{Im }\psi\equiv 0$
		on $\g$, hence $\int_\g ds=\int_\g \sqrt{2}|dh|=+\infty$ implies $p$ is a pole of $dh$.) The Riemann relation shows
		\begin{equation}
		2-2g=\sum_{i=1}^m \nu_i\geq 2m
		\end{equation}
		with $g$ the genus of $\bar{M}$, since $dh$ has no zero in $\bar{M}$. Therefore $g=0$, $m=1$ and $\nu_1=2$.
		We can take
		\begin{equation}
		M=\C,\quad \psi=f(z),\quad dh=\la dz
		\end{equation}
		with a polynomial $f$ and $\la\in \C\backslash \{0\}$. A direct calculation shows $M$ is an entire graph
		over the light-like plane spanned by $\ep_2$ and $\ep_3+\ep_4$.

	\end{proof}
	
\noindent{\bf Remarks:}
 \begin{itemize}
 \item The conclusion of Theorem \ref{graph3} does not hold for the non-algebraic case. An counterexample
	is given by the following W-data:
	\begin{equation}
	M=\C,\quad \psi=e^z,\quad dh=e^{-z}dz.
	\end{equation}

\item Combining Theorem \ref{graph4} and Theorem \ref{graph3} implies the total Gauss curvature of a nonflat, complete, algebraic,
degenerate stationary surface of parabolic type should be $\infty$. This conclusion can also be derived by
Ma-Wang-Wang's argument in \cite{M-W-W}, because of the existence of the bad singular end. Moreover, in virtue of Picard's big theorem,
we can drop off the 'algebraic' condition in the above result.
\end{itemize}

	
	Following Fujimoto's argument in \cite{F}, we can establish a Bernstein type theorem as follows.
	
	\begin{thm}
		\label{vdp}
		Let $M$ be a complete degenerate stationary surface of parabolic type, if the image $\psi_1$ (or $\psi_2$) omits at least $4$ points in $\C^*$, then
		$M$ has to be an affine space-like plane.
	\end{thm}
	
	\begin{proof}
		Without loss of generality we can assume $M$ is simply-connected, then $M$ is conformally equivalent to $\C$ or the unit disk $\Bbb{D}$.
		The Picard's theorem says the holomorphic function $\psi=\psi_1$ takes each point in $\C$ with the exception of at most 1 point, unless
		$\psi$ is constant. So it suffices the consider the case $M=\Bbb{D}$.
		
		Let $(\psi,dh)$ be the associated W-data of $M$, and the induced metric $ds^2$ is given in (\ref{metric5}). The completeness of $ds^2$ implies
		\begin{equation}\label{metric6}
		d\hat{s}^2:=(1+|\psi|^2)|dh|^2
		\end{equation}
		is also complete. Denote by $f$ a holomorphic function on $\Bbb{D}$ with no zero, so that $dh=fdz$. Assume $\psi$ omits $a_1,a_2,a_3\in \C$.
		Denote
		\begin{equation}
		E:=\{z\in \Bbb{D}:\psi'(z)=0\}.
		\end{equation}
If $E=\Bbb{D}$, then $M$ is flat (see Theorem \ref{t_r2}), causing a contradiction to the completeness of $M$. Thus $E$ is a discrete subset of $\Bbb{D}$.
		
		If $E=\emptyset$, we consider a many-valued function
		\begin{equation}\label{eta}
		\eta:=f^{\f{1}{1-p}}\psi'(z)^{-\f{p}{1-p}}\left(\prod\limits_{i=1}^3 (\psi-a_i)^{\f{kp}{1-p}}\right)
		\end{equation}
		on $\Bbb{D}$, where $p\in (0,1),k\in (0,1)$ are both constants to be chosen. Take an arbitrary single-valued branch of $\eta$, still dented by
		$\eta$ for the matter of convenience. Let
		\begin{equation}
		w=\mc{F}(z):=\int \eta dz
		\end{equation}
		be a holomorphic mapping from $\Bbb{D}$ into $\C$, satisfying $\mc{F}(0)=0$ and $\mc{F}'(z)=\eta(z)\neq 0$. Hence there exists
		a holomorphic inverse mapping $z=\mc{G}(w)$ on a neighborhood of $0$. Let $\Bbb{D}(R):=\{w:|w|<R\}$ is the largest ball that $\mc{G}$ can be defined,
		then $R<+\infty$ (otherwise, $\mc{G}$ is a nonconstant bounded entire function, which contracts to Liouville's theorem) and
		there exists $w_0$ on the boundary of $\Bbb{D}_R$, such that $\mc{G}$ cannot be extended beyond a neighborhood of $w_0$. Let
		$l:=\{tw_0:0\leq t< 1\}$ be the straight line segment starting from $0$ and limiting to $w_0$, then
		\begin{equation}
		\g:=\mc{G}(l)
		\end{equation}
		must be a divergent curve in $\Bbb{D}$. Observing that
		$$\f{dw}{dz}=f^{\f{1}{1-p}}\psi'(w)^{-\f{p}{1-p}}\left(\prod\limits_{i=1}^3 (\psi-a_i)^{\f{kp}{1-p}}\right)\Big(\f{dw}{dz}\Big)^{-\f{p}{1-p}}$$
		we get
		\begin{equation}
		\f{dz}{dw}=\psi'(w)^p\left(f^{-1}\prod\limits_{i=1}^3 (\psi-a_i)^{-kp}\right)\circ \mc{G}
		\end{equation}
		and the pull-back metric on $\Bbb{D}(R)$ is
		\begin{equation}\label{es1}\aligned
		\mc{G}^*d\hat{s}^2&=\left[\big((1+|\psi|^2)|f|^2\big)\circ \mc{G}\right]\Big|\f{dz}{dw}\Big|^2|dw|^2\\
		&=|\psi'(w)|^{2p}(1+|\psi(w)|^2)\prod\limits_{i=1}^3 |\psi(w)-a_i|^{-2kp}|dw|^2.
		\endaligned
		\end{equation}
		
		It is well-known that $\Bbb{D}(R)$ and $\Om:=\C\backslash \{a_1,a_2,a_3\}$ can be equipped with
		Riemannian metrics $g_P$ and $d\si^2$, respectively, so that both of them are complete Riemannian
		manifolds of $K\equiv -1$. More precisely,
		\begin{equation}
		g_P:=\f{4R^2}{(R^2-|w|^2)^2}|dw|^2
		\end{equation}
		is the Poincar\`{e} metric, and the universal covering map from $(\Bbb{D},g_P)$ onto
		$(\Om,d\si^2)$ is a local isometry. Denote
		\begin{equation}
		d\si^2:=\la^2 |d\xi|^2,
		\end{equation}
		then we have the asymptotic behavior on $\la$ as (see e.g. p.250 of \cite{Ne})
		\begin{equation}\label{es2}\aligned
		\la\sim \f{1}{|\xi-a_j|\log|\xi-a_j|^{-1}}\quad \text{near }&a_j\\
		\la\sim \f{1}{|\xi|\log|\xi|} \qquad\qquad
		\text{near }&\infty.
		\endaligned
		\end{equation}
		The Schwarz's lemma (see e.g. p. 13 of \cite{A}) says the holomorphic mapping $\psi\circ \mc{G}$ from $(\Bbb{D}(R),g_P)$
		into $(\Om,d\si^2)$ does not increase the distance, hence
		\begin{equation}\label{es3}
		\la(\psi(w))|\psi'(w)|\leq \f{2R}{R^2-|w|^2}.
		\end{equation}
		In conjunction with (\ref{es1}), (\ref{es2}) and (\ref{es3}) we obtain
		\begin{equation}
		\mc{G}^* d\hat{s}^2\leq \left(\f{2R}{R^2-|w|^2}\right)^{2p}|dw|^2
		\end{equation}
		whenever
		\begin{equation}\label{pk1}
		(3k-1)p>1.
		\end{equation}
		Thus
		\begin{equation}
		\int_\g d\hat{s}=\int_l \mc{G}^*d\hat{s}\leq \int_0^R \left(\f{2R}{R^2-r^2}\right)^p dr<+\infty.
		\end{equation}
		This contradicts to the completeness of $(M,d\hat{s}^2)$.
		
		For the case $E\neq \emptyset$, let $\td{M}:=\Bbb{D}$
be the universal covering space of $M_0:=\Bbb{D}\backslash E$,
 $\pi:\td{M}\ra M_0$
		be the universal covering map, $\td{\psi}:=\psi\circ \pi$ and $\pi^* dh$ be the pull-back W-data associated to the stationary
		immersion $\td{\mb{x}}:=\mb{x}|_{M_0}\circ \pi$. Considering the many-valued function
		\begin{equation}
		\td{\eta}:=\td{f}^{\f{1}{1-p}}\td{\psi}'(\td{z})^{-\f{p}{1-p}}\left(\prod\limits_{i=1}^3 (\td{\psi}-a_i)^{\f{kp}{1-p}}\right)
		\end{equation}
		with $\td{f}:=f\circ \pi$, we can proceed as above to find a divergent curve $\td{\g}:[0,1)\ra \td{M}$, such that
		$|\int_{\td{\g}}\td{\eta}d\td{z}|<+\infty$ and $\int_{\td{\g}}\pi^* d\hat{s}<+\infty$.
		Denote $\g:=\pi\circ \td{\g}$. If $\g$ is a divergent curve in $M$, then $\int_\g d\hat{s}<+\infty$, causing a contradiction to the
		completeness of $(M,d\hat{s}^2)$. Otherwise, $\g(t)$ converges to $z_0\in E$ as $t\ra 1^-$, and $|\int_\g \eta_\g dz|<+\infty$
		with $\eta_\g$ being a single-valued branch of $\eta$ (see (\ref{eta})) defined on a neighborhood of $\g$. On the other hand,
		since $\psi'(z_0)=0$, there exists $m\geq 1$, such that $\psi'(z)\sim (z-z_0)^m$ near $z_0$ and hence
		\begin{equation}
		|\eta_\g(z)|\sim |z-z_0|^{-\f{mp}{1-p}}.
		\end{equation}
		This implies $|\int_\g \eta_\g dz|=+\infty$ whenever
		\begin{equation}\label{pk2}
		\f{1}{2}<p<1
		\end{equation}
		and forces a contradiction. Combining (\ref{pk1}) and (\ref{pk2}), it suffices for us to choose
		$k:=\f{3}{4}$, $p:=\f{5}{6}$ to complete the proof of the present theorem.

	\end{proof}

 \noindent {\bf Remark. }Theorem \ref{graph4} presents us with examples of complete degenerate stationary surfaces of parabolic type, such that each $\psi_i$
 ($i=1$ or $2$) omits $1$ or $2$ points. However, it fails to use Voss's method \cite{V} to construct the example whose $\psi_i$ omits exactly
 3 points. More precisely, for $z_1,z_2\in \C$, let $\pi$ be the universal covering map from $\Bbb{D}$ onto $\C\backslash\{z_1,z_2\}$, and
 \begin{equation}
 M=\Bbb{D}, \psi=\pi, dh=\pi^*\left(\f{dz}{(z-z_1)(z-z_2)}\right),
 \end{equation}
 then it is easy to check that $d\hat{s}^2$ (see (\ref{metric6})) is a complete metric on $M$, but $ds^2$ (see (\ref{metric5})) is not.
 So there are 2 possibilities as follows: either we may construct another type of complete examples,
 such that the number of exception values of $\psi_i$ exactly equals $3$, or it is possible to make use of a more powerful strategy to
 obtain a more stronger Bernstein type result, replacing '4' of Theorem \ref{vdp} by '3'. This is a challenging problem
 that we would like to consider in the future.


\bigskip\bigskip
	\Section{Stationary surfaces with rational graphical Gauss image}{Stationary surfaces with rational graphical Gauss image}
	\label{S5}

\subsection{On solutions of $f(w)=\bar{w}$}
Inspired by the above results, we continue to study space-like stationary surfaces in $\R^{3,1}$ with {\it rational graphical
Gauss image}. Namely, the image of a such surface $M$ under the Gauss map $G$ lies in $\G\subset Q_{1,1}$, such that
\begin{equation}
\Psi(\G)=\G_f:=\{(w,f(w)):w\in \C^*\}\subset \C^*\times \C^*,
\end{equation}
where $f$ is a rational function of degree $m$. In other words, the components $\psi_1,\psi_2$ of the Gauss map satisfy a
prescribed restriction
\begin{equation}
\psi_2=f(\psi_1).
\end{equation}  	
In particular, $M$ is a degenerate (2-degenerate) stationary surface whenever $m\leq 1$ ($m=0$).

For $\si\in SO^+(3,1)$, let $\mc{M}=\mc{M}_T$ ($T\in SL(2,\C)$) be the corresponding M\"{o}bius transformation
(see \S \ref{conj}) and
\begin{equation}
\bar{\mc{M}}:=\mc{M}_{\bar{T}},
\end{equation}
then
\begin{equation}
\aligned
\Psi(\si(\G))=&\big\{(\mc{M}(w),\bar{\mc{M}}\circ f(w)):w\in \C^*\big\}\\
=&\big\{(w,\bar{\mc{M}}\circ f\circ \mc{M}^{-1}(w)):w\in \C^*\big\}\\
=&\G_{\bar{\mc{M}}\circ f\circ \mc{M}^{-1}}.
\endaligned
\end{equation}
This means, $\G_{1}=\Psi^{-1}(\G_{f_1})$ and $\G_{2}=\Psi^{-1}(\G_{f_2})$ are equivalent under the action of $SO^+(3,1)$, if and only
if $f_1$ and $f_2$ are {\it conjugate similar} to each other, i.e. there exists a M\"{o}bius transformation
$\mc{M}$ ensuring $f_2=\bar{\mc{M}}\circ f_1\circ \mc{M}^{-1}$.

Similarly to the degenerate cases, $\psi_2\neq \bar{\psi_1}$ implies $\psi_1$ cannot take values in
\begin{equation}\label{ef1}
E_f:=\{w\in \C^*:f(w)=\bar{w}\}.
\end{equation}
 It is easy to verify that
\begin{equation}
\mc{M}(E_f)=E_{\bar{\mc{M}}\circ f\circ \mc{M}^{-1}}.
\end{equation}
This means, for $f_1,f_2$ that are conjugate similar to each other,
the M\"{o}bius transformation gives a one-to-one correspondence between $E_{f_1}$ and $E_{f_2}$.

Now we write
\begin{equation}
f(w)=\f{P(w)}{Q(w)}
\end{equation}
with $P,Q$ being relatively prime polynomials of degree $\leq m$.
Let
\begin{equation}
F(w):=P(w)-\bar{w}Q(w)\qquad \forall w\in \C,
\end{equation}
be a so-called {\it polynomial bianalytic function} (see e.g. \cite{b}), then for each $w\in \C$,
$w\in E_f$ if and only if $w$ is a zero of $F$.
It is worthy to note that $F$ may have non-isolated zeros, e.g. when $P(w)=1, Q(w)=w$,
but this situation cannot happen whenever $m\geq 2$: let $F_1,F_2$ be respectively the real part and imaginary part of
$F$, then $\bar{F}\neq \pm F$ implies both $F_1$ and $F_2$ are nonzero polynomials of degree $\leq m+1$; assuming the existence of a nontrivial common factor
of $F_1$ and $F_2$ shall force a nontrivial common factor of $P$ and $Q$ and then cause a contradiction; thus, by
Bezout's theorem (see e.g. \S 2 of \cite{H-K-T}), the number of zeros of $F$, i.e. the common zeros of $F_1$ and $F_2$, cannot exceed $(m+1)^2$.
On the other hand, we can give an estimate for the lower bound of $|E_f|$ with the aid of the theory of differential topology:

\begin{pro}\label{ef}
If $m\geq 2$, then $E_f$ contains at least $m-1$ points.
\end{pro}
\begin{proof}
The easily-seen fact
$E_f\neq \C^*$ enables us to choose $\la\notin E_f$, and then letting $\mc{M}$ be a M\"{o}bius transformation that sends $\la$ to $\infty$
implies $\infty\notin E_{\bar{\mc{M}}\circ f\circ \mc{M}^{-1}}$. Thereby, we can assume $\infty\notin E_f$, i.e. $f(\infty)\neq \infty$
without loss of generality, and hence $\text{deg}(P)\leq \text{deg}(Q)=m$.
Now we write
$$Q(w)=\la_m w^m+\la_{m-1}w^{m-1}+\cdots+\la_1 w+\la_0$$
and let
\begin{equation}
\aligned
G(w):=&-\la_m w^{m-1}|w|^2\\
H(w,t):=&(1-t)F(w)+tG(w)\qquad \forall t\in [0,1].
\endaligned
\end{equation}
Then there exists $R_0>0$, such that
\begin{equation}
H(w,t)\neq 0\quad \text{on }\p \Bbb{D}(R)
\end{equation}
whenever $R\geq R_0$. This implies
\begin{equation}
\aligned
W(F|_{\p \Bbb{D}(R)},0)&=\text{deg}\left(\f{F}{|F|}\Big|_{\p \Bbb{D}(R)}\right)=\text{deg}\left(\f{G}{|G|}\Big|_{\p \Bbb{D}(R)}\right)\\
&=W(G|_{\p \Bbb{D}(R)},0)=m-1,
\endaligned
\end{equation}
where $W(\g,0):=\text{deg}(\f{\g}{|\g|})$ denotes the winding number of a close curve $\g$ with respect to $0$.

Let $a_1,\cdots,a_k$ be the zeros of $F$, i.e.
\begin{equation}
E_f=\{a_1,\cdots,a_k\}.
\end{equation}
For each $1\leq i\leq k$, let
\begin{equation}
\ep_i:=W(F|_{\p D_i},0)=\text{deg}\left(\f{F}{|F|}\Big|_{\p D_i}\right),
\end{equation}
where $D_i$ is an arbitrary simply-connected domain satisfying $D_i\cap E_f=\{a_i\}$. Since
$F$ has no zero on
\begin{equation}
V:=\Bbb{D}(R)\backslash \left(\bigcup_{i=1}^k \bar{D}_i\right),
\end{equation}
\S 5, Lemma 1 of \cite{M} tells us $\text{deg}(\f{F}{|F|}|_{\p V})=0$. Therefore
\begin{equation}\label{degree}
\sum_{i=1}^k \ep_i=\text{deg}\left(\f{F}{|F|}\Big|_{\p \Bbb{D}(R)}\right)=m-1.
\end{equation}

For each fixed $i$ being considered, we can assume $a_i=0$ without loss of generality, since the winding number is invariant under parallel translation.
Then $F(0)=0$ means $P(0)=0$ and $Q(0)\neq 0$, since $P$ and $Q$ are relatively prime. Without loss of generality we can assume
$Q(0)=1$ and $P'(0)=r\in \Bbb{R}^+\cup \{0\}$. (Otherwise we can replace $f$ by $e^{i\th}f$ with a suitable $\th$, which is
conjugate similar to $f$.) A direct calculation
shows $\ep_i=1$ whenever $r>1$ and $\ep_i=-1$ whenever $r\in [0,1)$. If $r=1$, then the imaginary part $F_2$ of $F$ can be written as
\begin{equation}
F_2(x,y)=2y+p(x,y),
\end{equation}
where $x$ and $y$ are respectively the real part and the imaginary part of $w$, and the degree of each term of $p(x,y)$ is no less than $2$.
We can find $\de_0>0$, such that for any $\de\in (0,\de_0]$,
\begin{itemize}
\item $F_2(x,\de)>0$ for all $x\in [-\de,\de]$;
\item $F_2(x,-\de)<0$ for all $x\in [-\de,\de]$;
\item There exists a unique $y_1\in (-\de,\de)$, such that $F_2(-\de,y_1)=0$;
\item There exists a unique $y_2\in (-\de,\de)$, such that $F_2(\de,y_2)=0$.
\end{itemize}
This means, for $D_\de:=(-\de,\de)\times (-\de,\de)$, $F(\p D_\de)$ intersects with the $x$-axis for exactly 2 times. Therefore
\begin{equation}
\ep_i=\lim_{\de\ra 0}W(F|_{\p D_\de},0)\in \{-1,0,1\}.
\end{equation}
And then $k\geq m-1$ immediately follows from (\ref{degree}).

\end{proof}
\noindent {\bf Remark. }If $f(w)=\f{1}{w^m}$, then $w\in E_f$ if and only if $w^{m-1}=1$, i.e. $E_f$ contains exactly $m-1$ points. This shows the estimate
of Proposition \ref{ef} is best.

Now we assume $M$ is complete and algebraic. Whenever $m\geq 2$, the proof of
Proposition \ref{ef} ensures the existence of $a_i\in E_f$, such that $\ep_i=1$. Without loss of generality we can assume $a_i=0$,
then $Q(0)=1$ and $P'(0)\in [1,+\infty)$ implies $f'(0)\neq 0$. Therefore, for $p\in \bar{M}$ satisfying $\psi_1(p)=a_i$,
the annular end $D\backslash \{p\}$ ($D$ is a sufficiently small neighborhood of $p$ in $\bar{M}$) has to a bad singular end of $M$, and
hence the total Gauss curvature of $M$ is $\infty$ (see \cite{M-W-W}). Moreover, in conjunction with Picard's big theorem we obtain the following
conclusion:
\begin{thm}
Each nonflat complete space-like stationary surface in $\R^{3,1}$ with rational graphic Gauss image and finite total Gauss curvature has to be a degenerate stationary surface of elliptic type or a minimal surface in $\R^3$.
\end{thm}

\subsection{Exceptional values of Gauss maps}
Proposition \ref{ef} enables us to find $\mu\in E_f$ provided that $m\geq 2$. Let $\mc{M}$ be the M\"{o}bius transformation satisfying $\mc{M}(\mu)=\infty$,
then $\infty\in E_{\bar{\mc{M}}\circ f\circ \mc{M}^{-1}}$. Thereby, without loss of generality we can assume
\begin{equation}\label{ef2}
\infty\in E_f.
\end{equation}
This means $f(\infty)=\infty$ and hence
\begin{equation}
m=\text{deg}(P)>\text{deg}(Q):=n.
\end{equation}
Denote
\begin{equation}
E_f=\{a_1,\cdots,a_l,\infty\}
\end{equation}
and let $b_1,\cdots,b_q$ be the zeros of $Q$ of the order $r_1,\cdots,r_q$, respectively, then
\begin{equation}\aligned
|f(w)-\bar{w}|&\leq C|w-a_i|\qquad\quad \text{near }a_i\ (1\leq i\leq l),\\
|f(w)-\bar{w}|&\leq C|w-b_j|^{-r_j}\qquad \text{near }b_j\ (1\leq j\leq q),\\
|f(w)-\bar{w}|&\leq C|w|^{m-n}\qquad\qquad \text{near }\infty
\endaligned
\end{equation}
and hence
\begin{equation}\label{f2}
|f(w)-\bar{w}|\leq C\left(\prod_{i=1}^l |w-a_i|\right)\left(\prod_{j=1}^q |w-b_j|^{-r_j}\right)\big(1+|w|^2\big)^{\f{m-l}{2}}.
\end{equation}
Let $(\psi_1,\psi_2,dh)$ be the W-data of $M$, then $\psi_2=f(\psi_1)$. (\ref{ef2}) says
\begin{equation}
\psi:=\psi_1
\end{equation}
is a holomorphic function. $p\in M$ is a zero of $dh$ if and only if $p$ is a pole of $\psi_2=f(\psi)$, with the same order. Hence
\begin{equation}
\om:=\prod_{j=1}^q (\psi-b_j)^{-r_j}dh
\end{equation}
is a holomorphic 1-form on $M$ with no zero. In conjunction with (\ref{f2}) and (\ref{phi5}), we see
\begin{equation}\label{metric7}
\aligned
d\hat{s}^2:=&\left(\prod_{i=1}^l |\psi-a_i|^2\right)\left(\prod_{j=1}^q |\psi-b_j|^{-2r_j}\right)\big(1+|\psi|^2\big)^{m-l}|dh|^2\\
=&\left(\prod_{i=1}^l |\psi-a_i|^2\right)\big(1+|\psi|^2\big)^{m-l}|\om|^2
\endaligned
\end{equation}
should be a complete metric on $M$.

Similarly as in \S \ref{parabolic}, we establish the following Bernstein type theorem for complete stationary surfaces with rational graphical Gauss image.
\begin{thm}\label{be1}
Let $M$ be a complete nonflat space-like stationary surface in $\R^{3,1}$, whose Gauss image is contained in the graph of a rational
function $f$ of degree $m$, i.e. $\psi_2=f(\psi_1)$, then the number of exceptional values of $\psi_1$ in $\C^*$ cannot exceed $m-|E_f|+3$, where
$E_f$ consists of all solutions of $f(w)=\bar{w}$.
\end{thm}
\noindent {\bf Remark. }For degenerate stationary surfaces, i.e. $m\leq 1$, the conclusion of Theorem \ref{be1} can be derived from
Theorem \ref{t_r2} ($m=0$, $|E_f|=1$), Theorem \ref{grph} ($m=1$, $|E_f|=2$), Theorem \ref{be2} ($m=1$, $|E_f|=0$) and Theorem \ref{vdp} ($m=1$,
$|E_f|=1$). Hence the following proof only concerns the cases of $m\geq 2$.

\begin{proof}
Without loss of generality we can assume $M$ is simply-connected. For $M=\Bbb{C}$, the result is a direct corollary of the Picard's theorem.
So it is sufficient to consider the case $M=\Bbb{D}$.

Denote $\om=gdz$, then $g$ is a holomorphic function on $\Bbb{D}$ with no zero. Besides all $l+1$ points in $E_f$, assume $\psi$ omits $c_1,\cdots,c_s\notin E_f$, so that
\begin{equation}
s+l+1\geq m-|E_f|+4=m-l+3,
\end{equation}
we shall consider
\begin{equation}
E:=\{z\in \Bbb{D}:\psi'(z)=0\},
\end{equation}
which cannot be the whole $\Bbb{D}$, since the nonflatness of $M$.

If $E=\emptyset$, let $\eta$ be a single-valued branch of the following many-valued function
\begin{equation}\label{mf}
z\in \Bbb{D}\mapsto g^{\f{1}{1-p}}\psi'(z)^{-\f{p}{1-p}}\left(\prod_{i=1}^l(\psi-a_i)^{\f{1+\be}{1-p}}\right)\left(\prod_{j=1}^s (\psi-c_j)^\f{\be}{1-p}\right),
\end{equation}
where $p\in (\f{1}{2},1)$, $\be\in (0,1)$ are constants to be chosen. Denote
\begin{equation}
\ze=\mc{F}(z):=\int \eta dz,
\end{equation}
then as in the proof of Theorem \ref{vdp}, we can find a positive number $R$ and $\ze_0\in \p \Bbb{D}(R)$, such that
$\mc{F}|_{\Bbb{D}(R)}$ has a holomorphic inverse map $z=\mc{G}(\ze)$, and
\begin{equation}
\g:t\in [0,1)\mapsto \mc{G}(t\ze_0)
\end{equation}
is a divergent curve in $\Bbb{D}$. A direct calculation based on (\ref{metric7}) and (\ref{mf}) yields
\begin{equation}\label{metric8}
\mc{G}^*d\hat{s}^2=\psi'(\ze)^{2p}(1+|\psi(\ze)|^2)^{m-l}\left(\prod_{i=1}^l |\psi(\ze)-a_i|^{-2\be}\right)\left(\prod_{j=1}^s|\psi(\ze)-c_j|^{-2\be}\right)|d\ze|^2.
\end{equation}
On the other hand, combining with the Schwarz's lemma and Navanlinna's theory, we get an estimate
\begin{equation}\label{es4}
\la(\psi(\ze))|\psi'(\ze)|\leq \f{2R}{R^2-|\ze|^2},
\end{equation}
where $\la$ is a position smooth function on $\C\backslash \{a_1,\cdots,a_l,c_1,\cdots,c_s\}$ satisfying
\begin{equation}\label{es5}\aligned
		\la\sim \f{1}{|\xi-a_i|\log|\xi-a_i|^{-1}}\quad \text{near }&a_i\\
       \la\sim \f{1}{|\xi-c_j|\log|\xi-c_j|^{-1}}\quad \text{near }&c_j\\
		\la\sim \f{1}{|\xi|\log|\xi|} \qquad\qquad\qquad
		\text{near }&\infty.
		\endaligned
		\end{equation}
In conjunction with (\ref{metric8}), (\ref{es4}) and (\ref{es5}) we have
\begin{equation}\label{es6}
\mc{G}^* d\hat{s}^2\leq \left( \f{2R}{R^2-|\ze|^2}\right)^{2p}|d\ze|^2
\end{equation}
whenever
\begin{equation}\label{p}
\be<p<-(m-l)+(l+s)\be.
\end{equation}
Now we take
\begin{equation}
\left\{\begin{array}{ll}
\be:=\f{1}{2}, p:=\f{2}{3} & \text{if }m=2,l=2,s=0\\
\be:=\f{s+l-2}{s+l-1}+\ep, p:=\f{s+l-2}{s+l-1}+2\ep & \text{otherwise}
\end{array}\right.
\end{equation}
with $\ep$ a sufficiently small positive number, then (\ref{p}) holds. By (\ref{es6}),
we see $\int_\g d\hat{s}<+\infty$, which causes a contradiction to the completeness of $(M,d\hat{s}^2)$.

For the case $E\neq \emptyset$, let $\pi$ be the universal covering map
from $\Bbb{D}$ onto $M_0:=\Bbb{D}\backslash E$, $\td{\psi}:=\psi\circ \pi$ and $\td{g}:=g\circ \pi$ be the lifting of $\psi$ and $g$, respectively, we consider
the many-valued function
\begin{equation}
\td{\eta}:=\td{g}^{\f{1}{1-p}}\td{\psi}'(\td{z})^{-\f{p}{1-p}}\left(\prod_{i=1}^l(\td{\psi}-a_i)^{\f{1+\be}{1-p}}\right)\left(\prod_{j=1}^s (\td{\psi}-c_j)^\f{\be}{1-p}\right).
\end{equation}
Similarly as above, we can find a divergent curve $\td{\g}:[0,1)\ra \td{M}$, satisfying $|\int_{\td{\g}}\td{\eta}d\td{z}|<+\infty$
and $\int_{\td{\g}}\pi^*d\hat{s}<+\infty$. If $\g:=\pi\circ \td{\g}$ is not a divergent curve, i.e. $\g(t)$ converges to $z_0\in E$
as $t\in 1^-$, we can derive a contradiction based on $p\in (\f{1}{2},1)$, as in the proof of Theorem \ref{vdp}. Hence $\g$
is a divergent curve in $\Bbb{D}$ with finite length, which also causes a contradiction. This completes the proof of Theorem \ref{be1}.
\end{proof}

In conjunction with Proposition \ref{ef}, we immediately get the following corollaries:
\begin{cor}
Let $M$ be a complete nonflat space-like stationary surface in $\R^{3,1}$. If the Gauss image of $M$ lies in the graph of a rational
function $f$ of degree $m$, then $|E_f|\leq \f{m+3}{2}$.
\end{cor}

\begin{cor}\label{be3}
Each complete space-like stationary surface in $\R^{3,1}$ whose Gauss image lies in the graph of a ration function of degree $m\geq 6$
has to be an affine space-like plane.
\end{cor}

\noindent{\bf Remarks:}
\begin{itemize}
\item Let $M$ be the universal covering space of $\{w\in \C^*:w^{m-1}\neq 1\}$ with $2\leq m\leq 5$ and $\pi$ be the holomorphic covering function,
$f(w):=\f{1}{w^m}$,
then $\mb{x}:M\ra \R^{3,1}$ generated by the W-data
\begin{equation}
\psi_1=\pi, \psi_2=f\circ \pi, dh=\pi^*\left(\f{w^m dw}{(w^{m-1}-1)^2(w-1)^{5-m}}\right)
\end{equation}
yields a complete space-like stationary surface. This shows the quantitative
condition of Corollary \ref{be3} is optimal.
\item On the contrary, let $f(w):=\f{1}{w^m}$ and $\pi: M\ra \C^*\backslash \{c_1,\cdots,c_k\}$ be the universal covering function, where
 $m\geq 1$, $1\leq k\leq m+3$ and $c_i\in \C$ for each $1\leq i\leq k$, then $\mb{x}^*:M\ra \R^4$ given by the W-data
 \begin{equation}
\psi_1^*=\pi, \psi_2^*=f\circ \pi, dh^*=\pi^*\left(\f{w^m dw}{(w-c_1)^{m+4-k}(w-c_2)\cdots (w-c_k)}\right)
\end{equation}
is a complete minimal surface. In fact, following the arguments of K. Voss \cite{V} and H. Fujimoto \cite{F}, we can prove:
{\it Let $M$ be a nonflat complete minimal surface in $\R^4$, such that $\psi_2^*=f(\psi_1^*)$ with a rational function $f$
of degree $m$, then $\psi_1^*$ can omit at most $m+3$ points in $\C^*$.} Therefore, the condition of Theorem \ref{be1} on the number of exceptional values are strictly weaker than the corresponding Bernstein type theorem for minimal surfaces in $\R^4$.

\end{itemize}

	\bigskip\bigskip
	
	\bibliographystyle{amsplain}

\end{document}